\documentclass{amsart}
\usepackage[utf8]{inputenc}

\usepackage[english]{babel}
\usepackage[utf8]{inputenc}
\usepackage{geometry}
\usepackage{amsmath,amsfonts,amssymb,amsthm} 
\usepackage{mathrsfs}  
\usepackage{braket} 
\usepackage{paralist}
\usepackage{graphicx}
\usepackage{placeins} 
\usepackage{booktabs}
\usepackage{xspace}
\usepackage{mathtools}
\usepackage{todonotes}
\usepackage{enumerate}

\usepackage{pgf,tikz}
\usepackage{pgfplots}
\usepackage{bm}
\usepackage{pstricks,pstricks-add,pst-plot}

\usepackage{color}
\usepackage{hyperref}
\hypersetup{urlcolor=blue, citecolor=blue}

\newcounter{cst}
\newcommand{\ctel}[1]{C_{\refstepcounter{cst}\label{#1}\thecst}}
\newcommand{\cter}[1]{C_{\ref{#1}}}

\def\magenta{}

\def\be{\begin{equation}}
\def\ee{\end{equation}}
\def\ov#1{\overline{#1}}
\def\grad{\nabla}
\def\div{\nabla\cdot}
\def\p{\partial}
\def\O{\Omega}
\def\k{\kappa}
\def\R{\mathbb{R}}
\def\Ee{{\mathcal E}}
\def\Dd{{\mathcal D}}

\def\Oo{{\mathcal O}}
\def\Tt{{\mathcal T}}
\def\Hh{{\mathcal H}}
\def\bHh{\boldsymbol{\mathcal H}}

\def\bbM{\mathbb{M}}
\def\wt{\widetilde}

\def\1{{\bf 1}}
\def\a{\alpha}
\def\b{\beta}
\def\eps{\varepsilon}
\def\G{\Gamma}
\def\sig{\sigma}

\def\brho{{\boldsymbol \rho}}
\def\bu{{\boldsymbol u}}

\def\bphi{{\boldsymbol \phi}}
\def\bxi{{\boldsymbol \xi}}

\def\bF{{\boldsymbol F}}
\def\bG{{\boldsymbol G}}

\usepackage[normalem]{ulem} 
\usepackage{cancel}
\def\Ff{\mathcal{F}}

\newtheorem{thm}{Theorem}[section]
\newtheorem{lem}[thm]{Lemma}
\newtheorem{prop}[thm]{Proposition}
\newtheorem{coro}[thm]{Corollary}
\newtheorem{Def}[thm]{Definition}

\newtheorem{rem}[thm]{Remark}

\newcommand{\oo}{\infty}

\newcommand{\logLogSlopeTriangle}[5]
{
	\pgfplotsextra
	{
		\pgfkeysgetvalue{/pgfplots/xmin}{\xmin}
		\pgfkeysgetvalue{/pgfplots/xmax}{\xmax}
		\pgfkeysgetvalue{/pgfplots/ymin}{\ymin}
		\pgfkeysgetvalue{/pgfplots/ymax}{\ymax}
		
		\pgfmathsetmacro{\xArel}{#1}
		\pgfmathsetmacro{\yArel}{#3}
		\pgfmathsetmacro{\xBrel}{#1-#2}
		\pgfmathsetmacro{\yBrel}{\yArel}
		\pgfmathsetmacro{\xCrel}{\xArel}
		
		\pgfmathsetmacro{\lnxB}{\xmin*(1-(#1-#2))+\xmax*(#1-#2)} 
		\pgfmathsetmacro{\lnxA}{\xmin*(1-#1)+\xmax*#1} 
		\pgfmathsetmacro{\lnyA}{\ymin*(1-#3)+\ymax*#3} 
		\pgfmathsetmacro{\lnyC}{\lnyA-#4*(\lnxA-\lnxB)}
		\pgfmathsetmacro{\yCrel}{\lnyC-\ymin)/(\ymax-\ymin)}
		
		\coordinate (A) at (rel axis cs:\xArel,\yArel);
		\coordinate (B) at (rel axis cs:\xBrel,\yBrel);
		\coordinate (C) at (rel axis cs:\xCrel,\yCrel);
		
		\draw[#5]   (A)-- node[pos=0.5,anchor=north] {\scriptsize{1}}
		(B)--
		(C)-- node[pos=0.,anchor=east] {\scriptsize{#4}} 
		(A);
	}
}

\date{}

\title[SQRA finite volumes for nonlinear drift-diffusion equations]{On the square-root approximation finite volume scheme for nonlinear drift-diffusion equations}
\author{Cl\'ement Canc\`es}
\address{Cl\'ement Canc\`es (\href{mailto:clement.cances@inria.fr}{\tt clement.cances@inria.fr}): Univ. Lille,  Inria, CNRS, UMR 8524  - Laboratoire Paul Painlevé, F-59000 Lille, France.}

\author{Juliette Venel}
\address{Juliette Venel (\href{mailto:juliette.venel@uphf.fr}{\tt juliette.venel@uphf.fr}): Univ. Polytechnique Hauts-de-France, INSA Hauts-de-France, CERAMATHS -- Laboratoire de Matériaux Céramiques et de Mathématiques, 
F-59313 Valenciennes, France.}

\begin{document}

\begin{abstract}
We study a finite volume scheme for the approximation of the solution to convection diffusion equations with nonlinear convection and Robin boundary conditions. The scheme builds on the interpretation of such a continuous equation as the hydrodynamic limit of some simple exclusion jump process. We show that the scheme admits a unique discrete solution, that the natural bounds on the solution are preserved, and that it encodes the second principle of thermodynamics in the sense that some free energy is dissipated along time. The convergence of the scheme is then rigorously established thanks to compactness arguments. Numerical simulations are finally provided, highlighting the overall good behavior of the scheme.
\end{abstract}

\maketitle

\section{Presentation of the problem}\label{sec:continuous}

\subsection{The governing equations}
In this paper, we focus on the simple yet already interesting nonlinear Fokker-Planck equation
\begin{subequations}\label{eq:cons}
\begin{align}
\p_t \rho + \div F = 0,& \label{eq:cons.01}\\
 F  + \eta(\rho) \grad \phi +  \grad \rho =  0, \label{eq:cons.02}&
\end{align}
\end{subequations}
set on a connected bounded open subset $\O$ of $\R^d$, which is further assumed to be polyhedral in what follows, 
and for positive times $t \geq 0$. 
Its (finite) Lebesgue measure is denoted by $m_\O$.
While diffusion is linear, convection is not since one considers a degenerate mobility function $\eta$ of the form
\be\label{eq:eta}
\eta(\rho) = \rho(1 - \rho)
\ee
accounting for volume-filling to enforce $0 \leq \rho \leq 1$. 
In~\eqref{eq:cons}, the potential $\phi \in W^{1,\oo}(\O)$ (referred as the electric potential in what follows) 
is assumed to be given, and nonnegative without loss of generality: $\phi \geq 0$.
Our purpose can be extended to the case of a self-consistent electric potential $\phi$ 
related to the charge density $\rho$ through a Poisson equation without other difficulties than those that 
are already addressed in the literature, see for instance~\cite{CCFG21}. 

The system we consider is not isolated as in~\cite{BDPS10}, but rather 
in interaction with a surrounding environment through its boundary $\G = \p\O$. More precisely, 
we assume that there exist $\a,\b\in W^{1,\infty}(\G)$ with $\a(x) > \b(x) > 0$ for all $x\in \G$ such that 
\be\label{eq:boundary.1}
F \cdot \nu = \a \rho - \b\quad \text{on}\; \R_+ \times \G,
\ee
where $\nu$ denotes the normal to $\G$ outward w.r.t. $\O$.
The system is complemented by an initial condition $\rho^0$ compatible with the volume-filling constraint:
\be\label{eq:init}
\rho_{|_{t=0}} = \rho^0 \in L^\oo(\O; [0,1]).
\ee
Our goal is to provide some provably convergent approximation of the problem~\eqref{eq:cons}--\eqref{eq:init}. 
The stability of our numerical method, to be detailed in Section~\ref{sec:scheme}, 
mimics some stability features of the continuous problem inherited from thermodynamics.

\subsection{Energy dissipation structure}\label{ssec:NRG}

The system~\eqref{eq:cons}--\eqref{eq:init} under consideration inherits some key property from thermodynamics. 
Defining its free energy by 
\[
\Ff(\rho) = \int_\O \left(h(\rho) + \rho \phi\right), \qquad h(\rho) = \rho \log(\rho) + (1-\rho) \log(1-\rho) + \log(2) \geq 0, 
\]
then it is dissipated within $\O$, but energy coming from the surrounding environment can enter the system 
thanks to the boundary flux~\eqref{eq:boundary.1}. 

Introducing the chemical and electrochemical potentials $\mu$ and $\xi$ respectively defined by 
\be\label{eq:mu+xi}
\mu = h'(\rho) = \log\frac\rho{1-\rho}, \qquad \xi = \mu + \phi = \frac{\delta \Ff}{\delta \rho}(\rho), \qquad 0 < \rho < 1, 
\ee
the chain rule $\grad \rho = \eta(\rho) \grad \mu$ allows to reformulate the flux 
\be\label{eq:flux.2}
F = -\eta(\rho) \grad( \phi + \mu) = -\eta(\rho) \grad \xi.
\ee
On the other hand, setting 
\be\label{eq:xiG}
\xi^\G = \phi - \log(\a/\b-1) \in W^{1,\infty}(\G) \quad \text{and}\quad \k = \sqrt{\b(\a-\b)} \in  W^{1,\infty}(\G),
\ee 
the boundary flux~\eqref{eq:boundary.1} can be expressed by the mean of a Butler-Volmer type formula:
\be\label{eq:boundary.2}
F \cdot \nu = \k \left(\rho e^{\frac12(\phi - \xi^\G)} - (1-\rho) e^{-\frac12(\phi - \xi^\G)} \right) = 
2 \k \sqrt{\rho(1-\rho)} \sinh\left(\frac12 (\xi - \xi^\G) \right).
\ee
The quantity $\xi^\G$ has to be thought as an electrochemical potential associated to the surrounding environment. 
When a quantity $n^\G = \int_\G F\cdot \nu$ of the chemical species of interest enters (resp. leaves) $\O$, the income (resp. loss) in free energy 
is equal to $n^\G \xi^\G$. Therefore, the total free energy defined (up to an additive constant) by 
\be\label{eq:Fftot}
\Ff_\text{tot}(t) = \Ff(\rho(t)) +  \int_0^t \int_\G \xi^\G F\cdot \nu, \qquad t \geq 0,
\ee
corresponds to the whole isolated system made of $\O$ and its surrounding environment. 
As the following proposition shows, it is decaying along time.

\begin{prop}\label{prop:NRG}
Let $\rho$ be a strong solution to~\eqref{eq:cons}--\eqref{eq:init}, then 
\be\label{eq:Fftot.decay}
\Ff_\text{tot}(t) \leq \Ff_\text{tot}(s) \leq \Ff(\rho^0) \leq \left(\|\phi\|_\infty + \log(2)\right) m_\O, \qquad t \geq s \geq 0. 
\ee
Moreover, there exists $\ctel{cte:Ftot}$ depending on $\G,\alpha, \beta$ and $\phi$ such that 
\be\label{eq:Ff.bound}
\Ff_\text{tot}(t)  \geq - \cter{cte:Ftot} t, \qquad t \geq 0.
\ee
\end{prop} 
\begin{proof}
On first remarks that thanks to its definition~\eqref{eq:Fftot}, the initial total free energy coincides with the free energy contained in $\O$, i.e. 
$\Ff_\text{tot}(0) = \Ff(\rho^0)$ thanks to~\eqref{eq:init}. The bound on the initial energy 
$\Ff(\rho^0)$ is readily deduced from $0 \leq \rho^0 \leq 1$ and $0 \leq h(\rho^0) \leq \log(2)$.
Let us now check that $\Ff_\text{tot}$ is decaying along time. To this end, let us compute 
\be\label{eq:dissip.cont}
\frac{d \Ff_\text{tot}}{dt}(t) = \int_\O \xi \p_t \rho +  \int_\G  \xi^\G F\cdot \nu = \int_\O F \cdot \grad \xi +   \int_\G (\xi^\G - \xi) F\cdot \nu. 
\ee
Both terms on the right-hand side are nonpositive respectively because of~\eqref{eq:flux.2} and \eqref{eq:boundary.2}, so that~\eqref{eq:Fftot.decay} 
holds true. 

To establish~\eqref{eq:Ff.bound}, one only has to notice that $\Ff(\rho(t))$ is nonnegative for all $t\geq 0$, so that 
\[
 \Ff_\text{tot}(t)  = \Ff(\rho(t)) +  \int_0^t \int \xi^\G F\cdot \nu \geq  \int_0^t \int \xi^\G F\cdot \nu  \geq 
 - t \ {\|\xi^\G\|}_\infty \left\| F\cdot \nu \right\|_\infty, \qquad t \geq 0.
\]
Uniform bounds on $\xi^\G$ and on $F\cdot \nu$ easily follow from their expressions~\eqref{eq:xiG} and~\eqref{eq:boundary.1} together with $0 \leq \rho \leq 1$.
\end{proof}

The estimates highlighted in Proposition~\ref{prop:NRG} encode some strong stability in the system \eqref{eq:cons}--\eqref{eq:init}.
The precise quantification of the dissipation rate of the total free energy even provides sufficiently compactness to establish the existence 
of weak solutions to~\eqref{eq:cons}--\eqref{eq:init}. The numerical method we introduce in the next section satisfies similar 
energy dissipation estimates, on which the numerical analysis we propose relies. 

\begin{Def}\label{def:weak}
A function $\rho$ is said to be a weak solution to~\eqref{eq:cons}--\eqref{eq:init} if:
\begin{enumerate}[(i)]
\item $\rho$ belongs to $L^\oo(\R_+\times\O; [0,1]) \cap L^2_\text{loc}(\R_+; H^1(\O))$, hence its trace  $\gamma\rho$  on $\R_+\times \G$ belongs to $L^\oo(\R_+\times\G; [0,1])\cap L^2_\text{loc}(\R_+; H^{1/2}(\G))$; 
\item for all $\varphi \in C^\infty_c([0,T) \times \ov \O)$, the following equality holds:
\be
\label{eq:weak}
\iint_{\R_+ \times \O} \rho \p_t \varphi + \int_\O \rho^0 \varphi(0,\cdot) - \iint_{\R_+ \times \O} \left( \eta(\rho) \grad \phi + \grad \rho \right) \cdot \grad \varphi 
- \iint_{\R_+\times \G} \left(\a\, \gamma \rho + \b\right) \varphi = 0.
\ee
\end{enumerate}
\end{Def}
\subsection{Goal and positioning of the paper}\label{ssec:goal}

The goal of this paper is to propose a seemingly new scheme to approximate nonlinear drift diffusion equations of the form~\eqref{eq:cons}. 
Such nonlinear drift diffusion problem arises in many contexts that are often more complex than the simple one prescribed by~\eqref{eq:cons}. 
We could for instance think about systems involving several species, coupled either via cross-diffusion~\cite{BDPS10}, or via a self-consistent 
electric potential~\cite{CCMRV}. We claim that a large part of our work (in practice all excepted what is related to uniqueness) can be transposed to the more 
complex setting of~\cite{CCMRV}. To enlighten the presentation, we rather adopt here a simpler setting, where 
the potential $\phi$ is given, but still with boundary conditions of Butler-Volmer type. 

Even though this scheme has a very natural probabilistic interpretation in terms of jump process, 
its use with a deterministic approach to compute solutions to~\eqref{eq:cons} 
has not been explored so far up to our knowledge. The scheme can be thought as an extension to the case of a nonlinear mobility function $\eta$ defined 
by~\eqref{eq:eta} of the approach proposed by~\cite{LFW13} and studied in~\cite{Heida18}, even though the method proposed therein is mesh-less and yields non-explicit diffusion tensors at the limit we avoid here (see also~\cite{HT_arXiv} for a mesh-based version of the scheme). Its analysis involves in particular some cosh-type dissipation potential, which have been shown recently in~\cite{MPPR17, Frenzel19, PS22} to appear in many contexts with strong connection with Boltzmann entropy. 

Our study covers several aspects. First, since our scheme is implicit, it yields a nonlinear system for which we show well-posedness and the preservation 
of the $L^\infty$ bounds. These properties follow from the monotonicity of the scheme. 
Another interesting aspect of the scheme is its free energy stability: a discrete counterpart to Proposition~\ref{prop:NRG} is established. 
Schemes encoding the second principle of thermodynamics have raised an important interest in the last years. In the case of a linear mobility $\eta(\rho) = \rho$, 
the Scharfetter-Gummel scheme~\cite{Chatard_FVCA6}, the SQRA scheme~\cite{Heida18}, 
or the Chang-Cooper scheme~\cite{BD10} are popular solutions since the scheme for solving the 
resulting linear Fokker-Planck equation amounts to the resolution of a linear system, in opposition to more involved strategies building on the Wasserstein gradient flow interpretation of the continuous problem (with no-flux boundary conditions), see for instance
\cite{MO14, BCL16, CGT20, LLW20, CCWW_FoCM}. 
{\magenta The Scharfetter-Gummel scheme has been extended to the context of nonlinear mobilities in~\cite{EFG06}, where the computation of the numerical flux requires the numerical resolution of some scalar nonlinear problem. 
The extension of the SQRA scheme proposed in this paper is more restrictive than the approach of \cite{EFG06} concerning the nonlinearities involved in the continuous problem, 
but the resulting expression for the numerical fluxes is explicit, making the scheme much cheaper.}
 
Second, we mathematically assess the convergence of the scheme when the discretization parameters (mesh size and time step) tend to 0. 
To this end, one needs to properly quantify the free energy dissipation. This is done thanks to primal and dual dissipation potentials inspired from~\cite{Mie11, MPR14}. 
The convergence proof then relies on compactness arguments, following the strategy of~\cite{EGH00}. Our convergence result is not quantitative, since no error estimate has been derived so far. Then we show in the numerical experiments that the scheme is second order accurate in space and first order in time.  See for instance \cite{HKS21}, where error estimates for several schemes including SQRA finite volumes are derived for steady linear Fokker-Planck equations. We also highlight the fact that 
the resolution of the nonlinear system by the Newton-Raphson method is efficient, even for large CFL conditions. 
The only drawback we have noticed so far for our scheme is its loss of accuracy in the large Péclet regim.

\section{The finite volume scheme and main results}\label{sec:scheme} 

Before introducing the so-called \emph{square-root approximation} (SQRA) scheme, one first needs to introduce some notation related 
to space and time discretizations. 

\subsection{Space and time discretizations of $\R_+ \times \O$}\label{ssec:mesh}
The SQRA finite volume scheme enters the framework of two-point flux approximation (TPFA) finite volumes, 
which are known to yield very efficient schemes but require meshes fulfilling the well-known orthogonality condition (iii) below, see for instance~\cite{Tipi, GK_Voronoi}.

\begin{Def}
\label{def:mesh}
An \emph{admissible mesh of $\O$} is a triplet $\left(\Tt, \Ee, {(x_K)}_{K\in\Tt}\right)$ such that the following conditions are fulfilled. 
\begin{enumerate}[(i)]
\item Each control volume (or cell) $K\in\Tt$ is non-empty, open, polyhedral and convex. We assume that 
\[
K \cap L = \emptyset \quad \text{if}\; K, L \in \Tt \; \text{with}\; K \neq L, 
\qquad \text{while}\quad \bigcup_{K\in\Tt}\ov K = \ov \O. 
\]
\item Each face $\sig \in \Ee$ is closed and is contained in a hyperplane of $\R^d$, with positive 
$(d-1)$-dimensional Hausdorff (or Lebesgue) measure denoted by $m_\sig = \Hh^{d-1}(\sig) >0$.
We assume that $\Hh^{d-1}(\sig \cap \sig') = 0$ for $\sig, \sig' \in \Ee$ unless $\sig' = \sig$.
For all $K \in \Tt$, we assume that 
there exists a subset $\Ee_K$ of $\Ee$ such that $\p K =  \bigcup_{\sig \in \Ee_K} \sig$. 
Moreover, we suppose that $\bigcup_{K\in\Tt} \Ee_K = \Ee$.
Given two distinct control volumes $K,L\in\Tt$, the intersection $\ov K \cap \ov L$ either reduces to a single face
$\sig  \in \Ee$ denoted by $K|L$, or its $(d-1)$-dimensional Hausdorff measure is $0$. 
\item The cell-centers $(x_K)_{K\in\Tt}$ are two by two distinct points of $\O$.  If $K, L \in \Tt$ 
share a face $K|L$, then the vector $x_L-x_K$ is orthogonal to $K|L$ and oriented from $K$ to $L$.
\item For the boundary faces $\sig \subset \p\O$, we assume  that 
there exists $x_\sig \in \sig$ such that $x_\sig - x_K$ is orthogonal to $\sig$.

\end{enumerate} 
\end{Def}
In the above definition, we do not suppose that $x_K$ belongs to $K$. We allow for more general grids, like for instance 
Delaunay triangulation or Laguerre cells. The condition on the fact that the $x_K$ are two-by-two distinct is not restrictive: if two cell centers 
$x_K$ and $x_L$ coincide, one just has to merge the two cells $K$ and $L$ and to remove $K|L$ from $\Ee$.

We denote by $m_K$ the $d$-dimensional Lebesgue measure of the control volume $K$.
The set of the faces is partitioned into two subsets: the set $\Ee_{\rm int}$ of the interior faces defined by 
\[
\Ee_{\rm int}= \left\{ \sig \in \Ee\; \middle| \; \sig = K|L\; \text{for some}\; K,L \in \Tt\right\}, 
\]
and the set $\Ee_{\rm ext}$ of the exterior faces defined by 
\[
\Ee_{\rm ext}= \left\{ \sig \in \Ee\; \middle| \; \sig \subset \p\O\right\}.
\]
For a given control volume $K\in\Tt$, we also define $\Ee_{K,{\rm int}} = \Ee_K \cap \Ee_\text{int}$ and $\Ee_{K,{\rm ext}} = \Ee_K \cap \Ee_\text{ext}$ the sets of its internal and external faces. 
We may write $\sigma=K|L$ to signify that $\sigma\in \Ee_{K,{\rm int}}$. For such internal edges $\sig = K|L$, we denote by $x_\sig$ the intersection between $[x_K,x_L]$ 
and the hyperplane containing $\sig$. Note that $x_\sig$ does not necessarily belong to $\sigma$.

In what follows, we denote by 
\[
d_\sig = \begin{cases}
|x_K - x_L| & \text{if}\; \sig = K|L \in \Ee_{\rm int}, \\
|x_K - x_\sig| & \text{if}\; \sig \in \Ee_{\rm ext}, 
\end{cases}
\qquad 
a_\sig = \frac{m_\sig}{d_\sig}, \qquad \sig \in \Ee.
\]
We also define the signed distance $d_{K\sig}$ between $x_K$ and $\sig \in \Ee_K$ thanks to the relation
\[
d_{K\sig} \nu_{K\sig} = x_\sig- x_K, \quad \sig \in \Ee_K, \, K\in\Tt, 
\]
where $\nu_{K\sig}$ stands for the normal to $\sig$ outward w.r.t. $K$.
Even though $d_{K\sig}$ can take negative values for interior faces, one still has 
\[
d_{K\sig} + d_{L\sig} = d_\sig >0 \quad \text{for}\; \sig = K|L \in \Ee_\text{int}, 
\]
as well as the geometric relation 
\[
m_K = \frac1d\sum_{\sig \in \Ee_K} m_\sig d_{K\sig}, \qquad K \in \Tt. 
\]
We further introduce the \emph{size} $\delta_\Tt$ and the \emph{regularity factor} $\zeta_\Tt$ 
of the mesh:
\be\label{eq:mesh.reg}
\delta_\Tt = \max_{K\in\Tt} \, \text{diam}(K), \qquad \zeta_\Tt = \max_{K\in\Tt} \max_{\sig \in \Ee_K} \left(\frac{\text{diam}(K)}{d_\sig} + \frac{d_\sig}{\text{diam}(K)}\right).
\ee
Given $\bu = \left(\left(u_K\right)_{K\in\Tt}, \left(u_\sig\right)_{\sig \in \Ee_\text{ext}}\right) \in \R^{\Tt \cup \Ee_\text{ext}}$, then for all $K\in\Tt$, we define the mirror value of $u_K$ w.r.t. $\sig \in \Ee_K$ by 
\be\label{eq:mirror}
u_{K\sig} = \begin{cases}
u_L & \text{if}\; \sig = K|L \in \Ee_\text{int}, \\
u_\sig & \text{if}\; \sig \in \Ee_\text{ext}.
\end{cases}
\ee

Concerning the time discretization, we consider for notation simplicity a uniform time stepping. 
More precisely, a time discretization is given by the choice of a time step $\tau>0$, from which we construct 
discrete times $t_n = n\tau$, $n\geq 0$. We stress that 
our study can be extended without any particular difficulty to the case of non-uniform time discretizations.

\subsection{The SQRA finite volume scheme}\label{ssec:scheme}

Given an admissible discretization $(\Tt, \Ee, (x_K)_{K\in\Tt})$ of $\O$ and a time step $\tau$, 
let us detail the scheme to be studied in this paper. 
First, the initial data $\rho^0$ is discretized into $\brho^0 = \left(\rho_K^0\right)_{K\in\Tt} \in [0,1]^\Tt$ 
by setting 
\be\label{eq:brho0}
\rho_K^0 = \frac1{m_K} \int_K \rho^0, \qquad K \in \Tt.
\ee
The potential $\phi$ is discretized into $\bphi = \left(\left(\phi_K\right)_{K\in\Tt}; \left(\phi_\sig\right)_{\sig\in\Ee_\text{ext}}\right)$  by setting 
\be\label{eq:bphi}
\phi_K = \phi(x_K) \quad \text{and} \quad \phi_\sig = \phi(x_\sig), \qquad K\in\Tt, \; \sig \in \Ee_\text{ext}.
\ee
As usual in the finite volume context, the conservation law \eqref{eq:cons.01} is discretized into 
\be\label{eq:cons.1}
\frac{\rho_K^n - \rho_K^{n-1}}{\tau} m_K + \sum_{\sig \in \Ee_K} m_\sig F_{K\sig}^n = 0, \qquad K \in \Tt, \; n \geq 1. 
\ee
The index $n$ for the numerical flux $F_{K\sig}^n$ across $\sig$ outward w.r.t. $K$ in~\eqref{eq:cons.1} indicates 
that our time discretization strategy relies on the backward Euler scheme. 
The bulk numerical fluxes are then defined by 
\begin{subequations}\label{eq:FKsig}
\be\label{eq:FKsig.int}
F_{K\sig}^n =  \frac1{d_\sig} \left[ \rho_K^n(1-\rho_L^n) e^{\frac12(\phi_K - \phi_L)} 
-  \rho_L^n(1-\rho_K^n) e^{\frac12(\phi_L - \phi_K)} \right] 
\quad \text{for}\; \sig = K|L \in \Ee_\text{int}. 
\ee
To preserve the second order accuracy in space, the boundary condition~\eqref{eq:boundary.1} is discretized by 
setting 
\be\label{eq:FKsig.ext}
F_{K\sig}^n = \frac1{d_\sig} \left[ \rho_K^n(1-\rho_\sig^n) e^{\frac12(\phi_K - \phi_\sig)} 
-  \rho_\sig^n(1-\rho_K^n) e^{\frac12(\phi_\sig - \phi_K)} \right] = \a_\sig \rho_\sig^n - \b_\sig, \quad \quad \text{for}\; \sig \in \Ee_\text{ext}, 
\ee
\end{subequations}
where, having set $\a_\sig = \a(x_\sig)$ and $\b_\sig = \b(x_\sig)$, 
\be\label{eq:rhosign}
\rho_\sig^n = \frac{d_\sig \b_\sig + \rho_K^n e^{\frac12(\phi_K-\phi_\sig)}}{d_\sig \a_\sig +  \rho_K^n e^{\frac12(\phi_K-\phi_\sig)} + (1-\rho_K^n)  e^{-\frac12(\phi_K-\phi_\sig)}}
\ee
is the unique value achieving the second equality in~\eqref{eq:FKsig.ext}. 
With a slight abuse of notation, we still denote by $\brho^n = \left(\left(\rho_K^n\right)_{K\in\Tt}, \left(\rho_\sig^n\right)_{\sig\in\Ee_\text{ext}}\right)$ the discrete density enriched with its boundary edge values prescribed by~\eqref{eq:rhosign}.

Formula~\eqref{eq:FKsig.int} can be interpreted as a Butler-Volmer law located 
at the interface between the cells $K$ and $L$. The probability that a particle jumps from $K$ to $L$ is 
proportional to the number $\rho_K^n$ of candidates in $K$ for a jump as well as to the number of 
available sites $(1-\rho_L^n)$ to host the particle in cell $L$. The drift $\phi_K - \phi_L$ appears in an exponential 
with balanced prefactors $1/2$, which is natural since $K$ and $L$ play symmetric roles in the formula. 
The scheme~\eqref{eq:cons.1}\&\eqref{eq:FKsig} is then a simple backward Euler discretisation of the dynamics 
prescribed by the infinitesimal generator of a weakly asymmetric simple exclusion process (WASEP), see~\cite{KL99}. 

Assume now that $\brho^n \in (0,1)^{\Tt\cup\Ee_\text{ext}}$ (this will be rigorously established later on, see Lemma~\ref{lem:Linf}). 
The consistency of formula~\eqref{eq:FKsig.int} with~\eqref{eq:cons.02} follows from the identity 
\be\label{eq:FKsig.int.2}
F_{K\sig}^n = \frac 2{d_\sig} \eta_\sig^n  \sinh\left(\frac{\xi^n_K-\xi^n_{K\sig}}2\right), 
\ee
with $ \xi_K^n = h'(\rho_K^n) + \phi_K$ for $K\in\Tt$, $\xi_\sig^n = h'(\rho_\sig^n) + \phi_\sig$ for $\sig \in \Ee_\text{ext}$, and where $\xi_{K\sig}^n$ is the mirror value of $\xi_K^n$ in the sense of~\eqref{eq:mirror}.
Moreover, we have set 
\be\label{eq:etasig}
\eta_\sig^n = \sqrt{\rho_K^n (1-\rho_K^n)\rho_{K\sig}^n(1-\rho_{K\sig}^n)} = \sqrt{\eta(\rho_K^n) \eta(\rho_{K\sig}^n)}, \qquad 
\ee
Taylor expanding formula~\eqref{eq:FKsig.int.2}, one gets that for each $n\geq 1$ and $\sig \in \Ee$, 
there exists $r_\sig^n \in (0,1)$ such that 
\begin{multline*}
F_{K\sig}^n = \frac{ \eta_\sig^n}{d_\sig}\Bigg( 2 \sinh\left(\frac{h'(\rho_K^n)-h'(\rho_{K\sig}^n)}2\right) 
+ (\phi_K - \phi_{K\sig}) \cosh\left(\frac{h'(\rho_K^n)-h'(\rho_{K\sig}^n)}2\right) \\
 + \frac{(\phi_K - \phi_{K\sig})^2}4  \sinh\left(\frac{h'(\rho_K^n)-h'(\rho_{K\sig}^n) + r_\sig^n(\phi_K-\phi_{K\sig})}2\right) 
\Bigg).
\end{multline*}
Then using the identities 
\begin{align}
\label{eq:id.sinh}
 \eta_\sig^n \sinh\left(\frac{h'(\rho_K^n)-h'(\rho_{K\sig}^n)}2\right)  =&\; \frac{\rho_K^n - \rho_{K\sig}^n }2, \\
 \label{eq:id.cosh}
 \eta_\sig^n \cosh\left(\frac{h'(\rho_K^n)-h'(\rho_{K\sig}^n)}2\right)  =&\; \frac{\eta(\rho_K^n) + \eta(\rho_{K\sig}^n)}2 
 + \frac{(\rho_K^n - \rho_{K\sig}^n)^2}2, 
\end{align}
and $\sinh(a+b) = \sinh(a) \cosh(b) + \sinh(b) \cosh(a)$, we get that 
\be\label{eq:FKsig.int.3}
F_{K\sig}^n = \frac{\rho_K^n - \rho_{K\sig}^n}{d_\sig} + \frac{\eta(\rho_K^n) + \eta(\rho_{K\sig}^n)}2  \frac{\phi_K - \phi_{K\sig}}{d_\sig}
+ R_\sig^n, 
\ee
with 
\begin{multline}\label{eq:Rsign}
R_\sig^n =  \frac{\phi_K - \phi_{K\sig}}{2 d_\sig} \left(\rho_K^n - \rho_{K\sig}^n\right)^2 
+ \frac{(\phi_K - \phi_{K\sig})^2}8 
\frac{\rho^n_K - \rho_{K\sig}^n}{d_\sig} \cosh\left(\frac{r_\sig^n(\phi_K-\phi_{K\sig})}2\right) \\+
 \frac{(\phi_K - \phi_{K\sig})^2}{8d_\sig}  \left(\eta(\rho_K^n) + \eta(\rho_{K\sig}^n) 
 + (\rho_K^n - \rho_{K\sig}^n)^2\right) \sinh\left(\frac{r_\sig^n(\phi_K-\phi_{K\sig})}2\right). 
\end{multline}
Let $\ov \rho: \R_+ \times \ov \O \to (0,1)$ be a smooth (say $C^{0,1}$ in time and $C^{1,1}$ in space) function, then for all $n\geq 1$, 
define $\ov \rho_K^n = \ov \rho(t_n, x_K)$,  $K\in\Tt$, and 
\[
\ov F_{K\sig}^n =  \frac1{d_\sig} \left[ \ov \rho_K^n(1-\ov \rho_L^n) e^{\frac12(\phi_K - \phi_L)}
-  \ov \rho_L^n(1-\ov \rho_K^n) e^{\frac12(\phi_L - \phi_K)}
\right], \qquad \sig = K|L \in \Ee_\text{int}.
\]
In the case of a uniform cartesian grid, where $x_\sig = \frac{x_K + x_L}2$ is the center of mass of $\sig$, 
it results from the expression~\eqref{eq:FKsig.int.3} of the flux that 
\[
\ov F_{K\sig}^n =  \frac1{m_\sig} \int_\sig \ov F(t_n) \cdot \nu_{K\sig} + \Oo(d_\sig^2),
\]
where $\ov F = - \grad \ov \rho - \eta(\ov \rho) \grad \phi$ is the flux corresponding to $\ov \rho$. The SQRA scheme, 
which owes its name to the choice~\eqref{eq:etasig} of a geometric mean for the edge mobilities $\eta_\sig^n$
and to the fact that it extends to the nonlinear mobility setting the linear SQRA scheme~\cite{LFW13, Heida18}, is then 
expected to be second order accurate w.r.t. space and first order accurate w.r.t. time since it relies on the backward Euler scheme. This will be confirmed by the numerical results exhibited in Section~\ref{sec:numerics}.

{\magenta
\begin{rem}\label{rem:equilibrium}
For general coefficient $\alpha$ and $\beta$ in \eqref{eq:boundary.1}, the system~\eqref{eq:cons} does not admit any thermal equilibrium, in 
the sense that there exists no steady profile $\rho^\infty$ such that $F \equiv 0$ in $\Omega$. Such a thermal equilibrium exists 
if and only if there exists some positive function $\lambda: \G \to (0,+\infty)$ and some constant $z\in \R$ such that 
\[
\alpha = \lambda(1 + e^{-\phi + z}), \qquad \beta = \lambda e^{-\phi + z} \quad \text{on} \; \G. 
\]
Then one readily checks that 
\[
\rho^\infty = \frac{e^{-\phi + z}}{1 + e^{-\phi + z}}
\]
is a thermal equilibrium corresponding to a constant electrochemical potential $\xi^\infty \equiv z$. 

Define now $\brho^\infty = \left(\rho_K^\infty\right)_{K\in\Tt}$ by setting 
\be\label{eq:rhoKinf}
\rho_K^\infty =  \frac{e^{-\phi_K + z}}{1 + e^{-\phi_K + z}}, \qquad K\in\Tt, 
\ee
then $\xi_K^\infty = z$ for all $K\in\Tt$. Owing to~\eqref{eq:FKsig.int.2}, the inner numerical fluxes all vanish. 
Moreover, 
\[
\rho_\sig^\infty =\frac{ \beta_\sigma }{ \alpha_\sigma} = \frac{e^{-\phi_\sigma + z}}{1 + e^{-\phi_\sigma + z}}, \qquad \sig \in \Ee_\text{ext},
\]
allows to solve the second equality in~\eqref{eq:FKsig.ext}, and the corresponding boundary fluxes $F_{K\sig}^\infty = 0$ for all $\sig \in \Ee_\text{ext}$.
In other words, $\brho^\infty$ given by~\eqref{eq:rhoKinf} is a discrete thermal equilibrium, and the scheme~\eqref{eq:bphi}--\eqref{eq:rhosign}
is well-balanced. 
\end{rem}
}
\subsection{Our main results and organisation of the paper}

Even though finer results can be found in the Sections \ref{sec:fixed} and \ref{sec:convergence} devoted to their proofs, 
we state here simple presentations of our main results. 
The first one, namely Theorem~\ref{thm:fixed}, is related to the characteristics of the scheme given a fixed mesh 
$(\Tt, \Ee, \left(x_K\right)_{K\in\Tt})$ and time step $\tau$. We show in particular that 
the scheme is well posed, preserves the $L^\infty$ bounds and is free-energy diminishing, in the sense that the discrete solution satisfies a discrete counterpart 
of Proposition~\ref{prop:NRG}. Then Theorem~\ref{thm:conv} states the convergence 
of the approximate solution provided by the scheme~\eqref{eq:brho0}--\eqref{eq:FKsig} towards the weak solution 
to~\eqref{eq:cons}--\eqref{eq:init} as the size of the mesh $\delta_\Tt$ and the time step $\tau$ tend to $0$.
The convergence analysis strongly relies on the energy stability of the scheme, and more precisely on
the quantification of the free energy dissipation. 

Given $\brho^n = \left(\rho_K^n\right)_{K\in\Tt} \in [0,1]^\Tt$, then we define 
\be\label{eq:FTt}
\Ff_\Tt(\brho^n) = \sum_{K\in\Tt} m_K \left(h(\rho_K^n) + \phi_K \rho_K^n \right), \qquad 
\Ff_{\Tt, \text{tot}}^n = \Ff_\Tt(\brho^n) + \sum_{p\geq 1} \tau \sum_{\sig \in \Ee_\text{ext}} m_\sig \xi^\G_\sig F_{K\sig}^p, 
\ee
where the external fluxes $F_{K\sig}^p$ are related to $\brho^p$ through formula~\eqref{eq:FKsig.ext}, 
and where, consistently with~\eqref{eq:xiG}, we have set 
\be\label{eq:xiGsig}
\xi_\sig^\G = \phi_\sig - \log\left(\frac{\a_\sig}{\b_\sig}-1\right), \qquad \sig \in \Ee_\text{ext}.
\ee
Initially, both energies coincide: 
 $\Ff_\Tt(\brho^0) = \Ff^0_{\Tt,\text{tot}}$, and is follows from Jensen's inequality and from the regularity of $\phi$ that
 \be\label{eq:Ff.init}
\Ff_\Tt(\brho^0) \leq  \Ff(\rho^0) + 2 \|\grad \phi \|_\infty \delta_\Tt m_\O.
 \ee
 In particular, $\Ff_\Tt(\brho^0)$ is bounded uniformly w.r.t. $\delta_\Tt$ owing to~\eqref{eq:Fftot.decay} and to $\delta_\Tt \leq \text{diam}(\O)$.
\begin{thm}\label{thm:fixed}
Given $\brho^{n-1} \in [0,1]^\Tt$, there exists a unique solution $\brho^{n} \in (0,1)^{\Tt\cup\Ee_\text{ext}}$ 
to the non\-linear system corresponding to the scheme~\eqref{eq:brho0}--\eqref{eq:rhosign}. Moreover, 
\be\label{eq:NRG.tot.Tt}
\Ff_{\Tt, \text{tot}}^{n-1} \geq \Ff_{\Tt, \text{tot}}^{n} \geq -  \cter{cte:Ftot} t_n, 
\qquad  n \geq 1,
\ee
with $\cter{cte:Ftot}$ as in Proposition~\ref{prop:NRG}.
\end{thm}
Theorem~\ref{thm:fixed} is a partial presentation of the results established in Section~\ref{sec:fixed}. 
Interested readers can find there some precise quantification of the dissipated total free energy we do not mention here to keep 
the presentation simple. 

Once Theorem~\ref{thm:fixed} and an iterated in time discrete solution $\left(\brho^n\right)_{n\geq 0}$ on hand, 
one can construct a piecewise constant in time and space reconstruction $\rho_{\Tt,\tau}$ by setting 
\be\label{eq:rhoTttau}
\rho_{\Tt,\tau}(t,x) = \rho_K^n \quad \text{if}\; (t,x) \in (t_{n-1},t_n] \times K,\; n \geq 1,  \qquad \rho_{\Tt,\tau}(0,x) = \rho_K^0 \quad \text{if}\; x \in K.
\ee
Now, let $\left(\Tt_m, \Ee_m, \left(x_K\right)_{K\in\Tt_m}\right)_{m\geq 0}$ and $\left(\tau_m\right)_{m\geq 0}$ be respectively a sequence of admissible meshes 
in the sense of Definition~\ref{def:mesh} and a sequence of time steps such that 
\be\label{eq:mtoinfty}
\lim_{m\to\oo} \delta_{\Tt_m} = \lim_{m\to\oo} \tau_m =0 \quad \text{and} \quad \zeta_{\Tt_m} \leq \zeta_\star <+\infty, \qquad m \geq 0, \
\ee
then the corresponding sequence of approximate solutions $\left(\rho_{\Tt_m,\tau_m}\right)_{m\geq 1}$ is bounded in $L^\oo(\R_+ \times \O)$ 
owing to Theorem~\ref{thm:fixed}. Therefore, there exists $\rho \in L^\oo(\R_+ \times \O)$ with $0 \leq \rho \leq 1$ such that, up to a subsequence, 
\be\label{eq:rho}
\rho_{\Tt_m,\tau_m} \underset{m\to\oo}\longrightarrow \rho \quad \text{in the $L^\oo(\R_+ \times \O)$ weak-$\star$ sense.}
\ee
The following theorem claim that $\rho$ is the unique weak solution to the continuous problem~\eqref{eq:cons}--\eqref{eq:init}, and 
that the convergence holds in a stronger sense. 
\begin{thm}\label{thm:conv}
Let $\rho$ be as in~\eqref{eq:rho}, then $\rho$ is the unique weak solution to \eqref{eq:cons}--\eqref{eq:init} in the sense of Definition~\ref{def:weak}. 
Moreover, the whole sequence $\left(\rho_{\Tt_m,\tau_m}\right)_{m\geq 0}$ converges strongly in $L^p_\text{loc}(\R_+ \times \ov \O)$ for any $p\in [1,+\infty)$.
\end{thm}
Proving Theorem~\ref{thm:conv} is the purpose of Section~\ref{sec:convergence}. The proof is based on compactness arguments 
that build on some refined version of the discrete energy estimate~\eqref{eq:NRG.tot.Tt}. Numerical evidences 
of the convergence will then be provided in Section~\ref{sec:numerics}. 
\section{Numerical analysis at fixed grid}\label{sec:fixed}

The goal of this section is twofold. First one aims at establishing Theorem~\ref{thm:fixed}. Second, one derives 
enough estimates to carry out the convergence analysis in Section~\ref{sec:convergence}. 

\subsection{Existence and uniqueness of the discrete solution}\label{ssec:well-posed}

We are interested in solutions $\brho^n$ to the scheme~\eqref{eq:cons.1}--\eqref{eq:rhosign} that 
are bounded between $0$ and $1$. Therefore, changing the definition~\eqref{eq:FKsig.int} of the internal fluxes 
by 
\begin{subequations}\label{eq:FKsig+}
\be\label{eq:FKsig+.int}
F_{K\sig}^n =\;  \frac1{d_\sig} \left[ \left(\rho_K^n\right)^+\left(1-\rho_L^n\right)^+ e^{\frac12(\phi_K - \phi_L)}
-  \left(\rho_L^n\right)^+\left(1-\rho_K^n\right)^+ e^{\frac12(\phi_L - \phi_K)} 
\right] \;\text{for}\; \sig = K|L \in \Ee_\text{int}, 
\ee
and the one~\eqref{eq:FKsig.ext} of the boundary fluxes by 
\be\label{eq:FKsig+.ext}
F_{K\sig}^n = \a_\sig \rho_\sig^n - \b_\sig \quad \text{with}\quad 
 \rho_\sig^n =  \frac{d_\sig \b_\sig + \left(\rho_K^n\right)^+  e^{\frac12(\phi_K-\phi_\sig)}}{d_\sig \a_\sig +  \left(\rho_K^n\right)^+ e^{\frac12(\phi_K-\phi_\sig)} + \left(1-\rho_K^n\right)^+  e^{-\frac12(\phi_K-\phi_\sig)}} \in (0,1)
\ee
\end{subequations}
does not affect the value of the solution $\brho^n$. After performing this slight modification, one can establish the 
following a priori estimate. 
\begin{lem}\label{lem:Linf}
Given $\brho^{n-1} \in [0,1]^\Tt$, $n\geq 1$, any solution $\brho^n$ to the modified scheme~\eqref{eq:cons.1}\&\eqref{eq:FKsig+}
belongs to $\left(0,1\right)^\Tt$. In particular, being a solution to~\eqref{eq:cons.1}\&\eqref{eq:FKsig+} is equivalent to being a solution in $(0,1)^\Tt$  to \eqref{eq:cons.1}--\eqref{eq:rhosign}.
\end{lem}
\begin{proof}
We argue by contradiction. Assume that there exists $K\in\Tt$ such that $\rho_K^n \geq 1$, then we deduce from~\eqref{eq:FKsig+} and from $\a_\sig > \b_\sig >0$ that $F_{K\sig}^n \geq 0$ for all $\sig \in \Ee_K$, 
and even that 
\be\label{eq:Linf.1}
F_{K\sig}^n > 0 \quad \text{if}\quad \sig \in \Ee_\text{ext}.
\ee
Then we deduce from \eqref{eq:cons.1} that 
\be\label{eq:Linf.2}
0 \leq \sum_{\sig \in\Ee_K} m_\sig F_{K\sig}^n = \frac{\rho_K^{n-1} - \rho_K^n}{\tau} m_K \leq 0. 
\ee
Therefore, $\rho_K^n = 1$ and all the fluxes $F_{K\sig}^n$, $\sig \in \Ee_K$ vanish. This is only possible if $\rho_L^n = 1$ for each neighboring cell $L$ such that 
$\sig = K|L \in \Ee_\text{int}$. One can iterate to neighbors of neighbors until reaching $K$ such that $\Ee_{K,\text{ext}} \neq \emptyset$. For such a cell, 
the first inequality in~\eqref{eq:Linf.2} is strict, leading to a contradiction, hence $\rho_K^n <1$ for all $K\in\Tt$. 
The bound $\rho_K^n >0$ for all $K\in\Tt$ can be established in a similar way.
\end{proof}

\begin{prop}\label{prop:well-posed}
For all $n \geq 1$, there exists a unique $\brho^n \in (0,1)^{\Tt}$ solution to \eqref{eq:brho0}--\eqref{eq:FKsig}. 
\end{prop}
\begin{proof}
The proof splits in two steps. Let us first show that at each time step $n\geq 1$, there 
exists a solution $\brho^n \in (0,1)^{\Tt}$ to \eqref{eq:cons.1}\&\eqref{eq:FKsig}, or equivalently owing to Lemma~\ref{lem:Linf}, $\brho^n$ 
solution to~\eqref{eq:cons.1}\&\eqref{eq:FKsig+}. Note that here, $\rho_\sig^n$ is thought as a function of $\rho_K^n$, cf.~\eqref{eq:rhosign},  rather than as an independent unknown. 

Let $n \geq 1$ be such that $\brho^{n-1} \in [0,1]^n$ is given (this is the case for $n=1$ owing to~\eqref{eq:brho0}). 
For $s \in [0,\tau]$, define $\brho^{(s)} = \left( \rho_K^{(s)}\right)_{K\in\Tt}$ as a solution to 
\be\label{eq:well-posed.1}
\left(\rho_K^{(s)} - \rho_K^{n-1}\right) m_K + s \sum_{\sig \in \Ee_{K}} m_\sig F_{K\sig}^{(s)} = 0, \qquad K\in\Tt, 
\ee
with $F_{K\sig}^{(s)}$ defined by~\eqref{eq:FKsig+} where $\brho^n$ has been replaced by $\brho^{(s)}$. 
For $s=0$, the above system {\magenta can be reformulated as} $\bbM \brho^{(0)} = \bbM \brho^{n-1}$, with the matrix $\bbM = \text{diag}\left(\left(m_K\right)_{K\in\Tt}\right)$ having 
a positive determinant. The unique solution $\brho^{(0)} =  \brho^{n-1}$ belongs to $[0,1]^\Tt$, while any solution $\brho^{(s)}$ for $s>0$ 
belongs to $(0,1)^\Tt$ thanks to Lemma~\ref{lem:Linf}. A standard topological degree argument (see \cite{LS34, Dei85} for a presentation of the topological gradient, and~\cite{EGGH98, Ahmed_intrusion} for applications in a similar context) then shows that the 
nonlinear system~\eqref{eq:well-posed.1} admits at least one solution $\brho^{(s)}$ for all $s>0$. In particular, for $s=\tau$, this shows the existence 
of $\brho^n\in (0,1)^\Tt$ solution \eqref{eq:cons.1}\&\eqref{eq:FKsig}. By a straightforward induction on $n$, one gets the existence of $\brho^n 
\in (0,1)^\Tt$ for all $n \geq 1$.

The second step of the proof consists in proving uniqueness for the solution in $(0,1)^\Tt$ to \eqref{eq:cons.1}\&\eqref{eq:FKsig}. 
Since ${\bf 0} \leq \brho^n \leq {\bf 1}$, $F_{K\sig}^n$ is an increasing function of $\rho_K^n$ and a non-increasing one 
of $\rho_L^n$ for  $L\neq K$. As a consequence, the nonlinear system corresponding to \eqref{eq:cons.1} {\magenta can be rewritten as follows:}
\be\label{eq:bHh}
\bHh^n(\brho^n) = \left( \Hh_K^n\left(\rho_K^n, \left(\rho_L^n\right)_{L\neq K}\right) \right)_{K\in\Tt} = {\bf 0}, 
\ee
where $\Hh_K$ is increasing w.r.t. its first variable and non-decreasing w.r.t. the others. 
Assume that the scheme admits another solution $\check \brho^n \in [0,1]^\Tt$ corresponding to the same previous step data $\brho^{n-1}$: 
\[
\bHh^n(\check \brho^n) = \left( \Hh_K^n\left(\check \rho_K^n, \left(\check \rho_L^n\right)_{L\neq K}\right) \right)_{K\in\Tt} = {\bf 0}, 
\]
Therefore, denoting by $a \wedge b = \min(a,b)$ and $a\vee b = \max(a,b)$, one has 
\[
\Hh_K^n\left(\rho_K^n, \left(\rho_L^n \wedge \check \rho_L^n \right)_{L\neq K}\right) \geq 0, 
\qquad 
\Hh_K^n\left(\check \rho_K^n, \left(\rho_L^n \wedge \check \rho_L^n \right)_{L\neq K}\right) \geq 0, 
\]
so that, since $\rho_K^n \wedge \check \rho_K^n$ is either equal to $\rho_K^n$ or $\check \rho_K^n$, 
\be\label{eq:well-posed.min}
\Hh_K^n\left( \rho_K^n \wedge \check \rho_K^n, \left(\rho_L^n \wedge \check \rho_L^n \right)_{L\neq K}\right) \geq 0, \qquad K\in\Tt. 
\ee
Similarly, there holds 
\be\label{eq:well-posed.max}
\Hh_K^n\left( \rho_K^n \vee \check \rho_K^n, \left(\rho_L^n \vee \check \rho_L^n \right)_{L\neq K}\right) \leq 0, \qquad K\in\Tt. 
\ee
Subtracting~\eqref{eq:well-posed.min} to \eqref{eq:well-posed.max}, summing over $K\in\Tt$ and using the conservativity of the fluxes provides 
\[
\sum_{K\in\Tt} \frac{\left|\rho_K^n - \check \rho_K^n\right|}\tau m_K + \sum_{\sig \in \Ee_\text{ext}} m_\sig \a_\sig |\rho_\sig^n - \check \rho_\sig^n| 
 \leq 0, 
\]
where $\check \rho_\sig^n$ is computed thanks to formula~\eqref{eq:rhosign} with $\check \rho_K^n$ instead of $\rho_K^n$. 
We conclude that $\brho^n = \check \brho^n$, completing the proof of Proposition~\ref{prop:well-posed}.
\end{proof}

\subsection{Discrete energy/dissipation estimates}\label{ssec:NRG.disc}
The goal of this section is to show some refined energy estimate implying in particular~\eqref{eq:NRG.tot.Tt}. 
We pay attention to precisely quantifying the energy dissipation since this information is key to derive 
the compactness results to be used in Section~\ref{sec:convergence}. 
Denote by $\bF^n = \left(F_{K\sig}^n\right)_{K \in \Tt, \sig \in \Ee_K}$ the approximate fluxes at time step $n\geq 1$, then 
taking inspiration in~\cite{MPR14, PRST22}, we introduce the primal dissipation potential by setting 
\be\label{eq:Dd}
\Dd_\Ee(\brho^n, \bF^n) = \sum_{\sig \in \Ee} a_\sig \eta_\sig^n\Psi\left( \frac{d_\sig F_{K\sig}^n}{\eta_\sig^n}\right) \geq 0, 
\ee
where $\Psi$ is the continuous nonnegative strictly convex even function vanishing at $0$ with superlinear growth at $\oo$ 
defined by 
\[
\Psi(z) = 2 z \log \left( \frac{z + \sqrt{z^2 + 4}}2\right) - 2 \sqrt{z^2+4} + 4, \qquad z \in \R, 
\]
and where $\eta_\sig^n$, which is defined by~\eqref{eq:etasig} for $\sig \in \Ee_\text{int}$ and by 
$\eta_\sig^n = \sqrt{\eta(\rho_K^n) \eta(\rho_\sig^n)}$ for $\sig \in \Ee_\text{ext}$,
is positive thanks to Proposition~\ref{prop:well-posed}. 

As highlighted by the notation, the dissipation 
is associated to the edges $\Ee$. Yet, the dissipation potential defined in~\eqref{eq:Dd} only corresponds to the dissipation in the bulk even though boundary fluxes also contribute to the dissipation of the total free energy, as shows~\eqref{eq:dissip.cont}. This choice is made for simplicity and is possible since 
the quantification of the dissipation across the sole bulk already provides enough compactness to carry out the 
convergence proof, see Section~\ref{sec:convergence}. Note that each internal edge $\sig = K|L \in \Ee_\text{int}$ appears 
only once in~\eqref{eq:Dd} and that $\Psi(\frac{d_\sig F_{K\sig}^n}{\eta_\sig^n}) = \Psi(\frac{d_\sig F_{L\sig}^n}{\eta_\sig^n})$ since $\Psi$ is even and since $F_{K\sig}^n + F_{L\sig}^n = 0$.

Given $\bG^n = \left(G_{K\sig}^n\right)_{K\in\Tt, \sig \in \Ee_K}$ with $G_{K\sig}^n + G_{L\sig}^n = 0$ for all $\sig = K|L \in \Ee_{\text{int}}$, 
then we define the  dual dissipation potential $\Dd_\Ee^*: (0,1)^\Tt \times \R^\Ee \to \R_+$ by 
\be\label{eq:Dd*}
\Dd_\Ee^*(\brho^n, \bG^n) = \sum_{\sig \in \Ee} a_\sig  \eta_\sig^n \Psi^*( G_{K\sig}^n) \geq 0, 
\ee
where $\Psi^*$ is the Legendre transform of $\Psi$, defined by 
\[
\Psi^*(s) = 4 \left(\cosh(s/2) - 1\right),
\qquad s \in \R.
\]
It is continuous, nonnegative, uniformly convex and vanishes at $0$. 
\begin{prop}\label{prop:NRG.disc}
Let $\left(\brho^n\right)_{n\geq 1} \subset (0,1)^{\Tt\cup\Ee_\text{ext}}$ be the iterated solution to the scheme~\eqref{eq:brho0}--\eqref{eq:rhosign}. For $n\geq 1$, 
let $\bG^n = \left(G_{K\sig}^n\right)_{K\in\Tt, \sig \in \Ee_K}$  be defined by 
\be\label{eq:GKsig}
G_{K\sig}^n = \xi_K^n - \xi_{K\sig}^n =  \begin{cases}
\xi_K^n - \xi_L^n & \text{if}\; \sig = K|L \in \Ee_\text{int}, \\
\xi_K^n - \xi_\sig^n & \text{it}\; \sig \in \Ee_\text{ext}, 
\end{cases}
\ee
then there holds
\be\label{eq:NRG.disc}
\frac{\Ff_{\Tt,\text{tot}}^n - \Ff_{\Tt,\text{tot}}^{n-1}}{\tau} + \Dd_\Ee(\brho^n, \bF^n) + \Dd_\Ee^*(\brho^n, \bG^n) \leq 0.
\ee
\end{prop}
\begin{proof} 
Since the solution $\brho^n$, $n\geq1$, belongs to $(0,1)^{\Tt\cup\Ee_\text{ext}}$, the discrete electrochemical potential 
$\bxi^n = \left(\left(\xi_K^n\right)_{K\in\Tt}, \left(\xi_\sig^n\right)_{\sig \in \Ee_\text{ext}}\right) \in \R^{\Tt\cup\Ee_\text{ext}}$ is well defined. 
Multiplying the discrete conservation law~\eqref{eq:cons.1} by $\xi_K^n$ and summing over $K\in\Tt$ leads to 
\be\label{eq:ABC}
A_\Tt^n + B_\Tt^n + C_\Tt^n= 0, 
\ee 
with
\[
A_\Tt^n =  \sum_{K\in\Tt} m_K \frac{\rho_K^n - \rho_K^{n-1}}\tau \xi_K^n , \qquad
B_\Tt^n =   \sum_{\sig \in \Ee}  m_\sig F_{K\sig}^n G_{K\sig}^n
\quad \text{and} \quad 
C_\Tt^n =  \sum_{\sig \in \Ee_\text{ext}} m_\sig F_{K\sig}^n \xi_\sig^n.
\]
Similarly to what we did in~\eqref{eq:boundary.2} at the continuous level, the external fluxes $F_{K\sig}^n$ given by~\eqref{eq:FKsig.ext} 
can be rewritten as
\[
F_{K\sig}^n = 2 \sqrt{\beta_\sig(\a_\sig - \beta_\sig) \rho_\sig^n (1-\rho_\sig^n)} \sinh\left(\frac12(\xi_\sig^n - \xi_\sig^\G)\right), \qquad \sig \in \Ee_\text{ext}.
\]
Therefore, $F_{K\sig}^n (\xi_\sig^n - \xi_\sig^\G) \geq 0$ for all $\sig \in \Ee_\text{ext}$, so that
\be\label{eq:C}
C_\Tt^n \geq  \sum_{\sig \in \Ee_{\text{ext}}} m_\sig F_{K\sig}^n  \xi_\sig^\G. 
\ee
Concerning the bulk term $B_\Tt^n$, the writing~\eqref{eq:FKsig.int.2} of the internal edge fluxes and its straightforward counterpart for boundary edges show that 
\[
\frac{d_\sig F_{K\sig}^n}{\eta_\sig^n} = 2 \sinh\left(\frac{G_{K\sig}^n}2\right) = \left(\Psi^*\right)'(G_{K\sig}^n).
\]
Therefore, we have equality in the Young-Fenchel inequality 
\[
\frac{d_\sig F_{K\sig}^n}{\eta_\sig^n} G_{K\sig}^n = \Psi\left(\frac{d_\sig F_{K\sig}^n}{\eta_\sig^n}\right) + \Psi^*\left(G_{K\sig}^n\right). 
\]
As a consequence, 
\be\label{eq:B}
B_\Tt^n =   \sum_{\sig \in \Ee}  a_\sig \eta_\sig^n \frac{d_\sig F_{K\sig}^n}{\eta_\sig^n} G_{K\sig}^n
= \Dd_\Ee(\brho^n, \bF^n) + \Dd_\Ee^*(\brho^n,\bG^n).
\ee
For the accumulation term $A_\Tt^n$, the convexity of the mixing entropy density $h$ implies that 
\[
\left(\rho_K^n - \rho_K^{n-1}\right) h'(\rho_K^n) \geq h(\rho_K^n) - h(\rho_K^{n-1}), \qquad K\in\Tt. 
\]
Therefore, it follows from the definition~\eqref{eq:FTt} of $\Ff_\Tt(\brho^n)$ that 
\be\label{eq:A}
A_\Tt^n \geq \frac{\Ff_\Tt(\brho^n) - \Ff_\Tt(\brho^{n-1})}\tau.
\ee
We recover the discrete energy dissipation estimate~\eqref{eq:NRG.disc} by combining~\eqref{eq:C}--\eqref{eq:A} in~\eqref{eq:ABC} and by using the definition~\eqref{eq:FTt} of $\Ff_{\Tt,\text{tot}}^n$.
\end{proof}

Since $\phi$ is assumed to be nonnegative, so does $\Ff_\Tt(\brho^n)$. The upper bound we deduce from Proposition~\ref{prop:NRG.disc} 
is rather on $\Ff_{\Tt, \text{tot}}^n$, which is not bounded from below so far. Obtaining a time-dependent lower-bound for $\Ff_{\Tt, \text{tot}}^n$ is the purpose 
of the following corollary. Its proof, the details of which are left to the reader, relies on the fact that both $\xi^\G_\sig$ and $F_{K\sig}^n$ are uniformly bounded 
for each $\sig \in \Ee_{\text{ext}}$.

\begin{coro}\label{coro:NRG.disc}
Let $\cter{cte:Ftot}$ be as in Proposition~\ref{prop:NRG}, then 
\[
- \cter{cte:Ftot} t_n \leq \Ff_{\Tt, \text{tot}}^n \leq \Ff_{\Tt, \text{tot}}^{n-1} \leq \Ff_\Tt(\brho^0) ,
\qquad  n \geq 1.
\]
In particular, \eqref{eq:NRG.tot.Tt} holds true.
\end{coro}

In Proposition~\ref{prop:NRG.disc}, the free energy dissipation is quantified thanks to the non-homogeneous functionals $\Dd_\Ee$ and $\Dd_\Ee^*$. 
The goal of the next Lemma is to deduce from this estimate some more classical discrete $L^2_\text{loc}(\R_+; H^1(\O))$ estimate on $\left(\brho^n\right)_{n\geq1}$.

\begin{lem}\label{lem:dissip*}There exists $\ctel{cte:L2H1}$ depending only on $\cter{cte:Ftot}$, $\O$ and $\phi$  such that
\[
\sum_{p=1}^n \tau \sum_{\sig \in \Ee} a_\sig \left( \rho_K^p - \rho_{K\sig}^p \right)^2 \leq \cter{cte:L2H1}(1 + t_n), \qquad \forall\ n \geq 1. 
\]
\end{lem}
\begin{proof}
Combining Proposition~\ref{prop:NRG.disc} with Corollary~\ref{coro:NRG.disc}, we obtain that 
\be\label{eq:dissip*.1}
\tau \sum_{p = 1}^n \Dd_\Ee^*(\brho^p, \bG^p)
= \sum_{p = 1}^n \tau \sum_{\sig \in \Ee} a_\sig \eta_\sig^p  \Psi^*(G_{K\sig}^p) \leq \Ff_{\Tt, \text{tot}}^0 - \Ff_{\Tt, \text{tot}}^n \leq  \Ff_\Tt(\brho^0) + \cter{cte:Ftot} t_n. 
\ee
It follows from the elementary inequality $\cosh(a+b) = \cosh(a) \cosh(b) + \sinh(a) \sinh(b)$ that 
\begin{multline*}
\eta_\sig^p  \Psi^*(G_{K\sig}^p) = 4 \eta_\sig^p \left(\cosh\left(\frac{\phi_K-\phi_{K\sig}}2\right)  \cosh\left(\frac{h'(\rho_K^p)-h'(\rho_{K\sig}^p)}2\right) - 1 \right) \\
+ 4 \eta_\sig^p \sinh\left(\frac{\phi_K-\phi_{K\sig}}2\right)  \sinh\left(\frac{h'(\rho_K^p)-h'(\rho_{K\sig}^p)}2\right)=: S_{\sig}^p + T_\sig^p, \qquad p \geq 1. 
\end{multline*}
Then using~\eqref{eq:id.cosh}, $\cosh(a) \geq 1$, and the fact that the arithmetic mean is greater than the geometric one, one gets that 
\[
 S_{\sig}^p = 2 \cosh\left(\frac{\phi_K-\phi_{K\sig}}2\right) \left( (\rho_K^p - \rho_{K\sig}^p)^2 + \eta(\rho_K^p) + \eta(\rho_{K\sig}^p) - 2 \eta_\sig^p\right) \geq 2 (\rho_K^p - \rho_{K\sig}^p)^2.
\]
On the other hand, \eqref{eq:id.sinh} yields 
\[
T_{\sig}^p = 2 (\rho_K^p - \rho_{K\sig}^p) \sinh\left(\frac{\phi_K-\phi_{K\sig}}2\right). 
\]
Since 
\[
\left|2 \sinh\left(\frac{\phi_K-\phi_{K\sig}}2\right)\right| \leq \cosh\left(\frac{{\|\phi\|_\infty}}2\right) |\phi_K - \phi_{K\sig}| \leq  \cosh\left(\frac{{\|\phi\|_\infty}}2\right) {\| \grad \phi\|}_\infty d_\sig, 
\]
we deduce from Young's inequality that 
\[
T_\sig^p \geq - (\rho_K^p - \rho_{K\sig}^p)^2 -  \cosh^2\left(\frac{{\|\phi\|_\infty}}2\right) {\| \grad \phi\|}^2_\infty d_\sig^2.
\]
All in all, we obtain 
\begin{align*}
\tau \sum_{p = 1}^n \Dd_\Ee^*(\brho^p, \bG^p) \geq& \;  \sum_{p=1}^n \tau \sum_{\sig \in \Ee} a_\sig \left( \rho_K^p - \rho_{K\sig}^p \right)^2 
-  \cosh^2\left(\frac{{\|\phi\|_\infty}}2\right) {\| \grad \phi\|}^2_\infty \sum_{p=1}^n \tau \sum_{\sig \in \Ee} m_\sig d_\sig \\
= & \; \sum_{p=1}^n \tau \sum_{\sig \in \Ee} a_\sig \left( \rho_K^p - \rho_{K\sig}^p \right)^2  -  \cosh^2\left(\frac{{\|\phi\|_\infty}}2\right) {\| \grad \phi\|}^2_\infty d m_\O t_n, 
\end{align*}
which provides the desired result after being combined with~\eqref{eq:dissip*.1}.
\end{proof}

Next lemma exploits the other part of the dissipation to derive some discrete $W^{1,1}_\text{loc}(\R_+; W^{-1,1}(\O))$ on $\left(\brho^n\right)_{n\geq 0}$.
\begin{lem}\label{lem:dissip}
Let $\varphi \in C^\infty_c([0,T)\times \O)$ for some $T>0$, with $\text{dist}(\text{supp}\ \varphi, \, \p\O) \geq \zeta_\star\delta_\Tt$, then define 
$\varphi_K^n= \frac1{m_K}\int_K \varphi(t_n)$ for all $K\in\Tt$ and $n\geq 0$, then there exists $\ctel{cte:L1W-11}$ 
depending only on $\O$, $\alpha$, $\beta$, $\phi$, $T$, and $\zeta_\star$ such that 
\[
\sum_{n\geq 1} \sum_{K\in\Tt} m_K \left(\rho_K^n - \rho_K^{n-1} \right) \varphi_K^{n} \leq \cter{cte:L1W-11} {\|\grad \varphi\|}_\infty. 
\]
\end{lem}
\begin{proof}
The assumption on the support of $\varphi$ implies that $\varphi_K^n = 0$ either if $K$ has a boundary edge $\sig \in \Ee_{K,\text{ext}}$ or 
if $n \geq T/\tau$. Moreover, it follows from the mean value theorem that for all $K\in \Tt$ and all $n \geq 1$, there exists $y_K^n \in K$ such that 
$\varphi_K^n = \varphi(t_n,y_K^n)$. Then for all $\sig = K|L \in \Ee_\text{int}$, one has 
\[
\left|\varphi_K^n - \varphi_L^n\right| \leq {\|\grad \varphi\|}_\infty \left(|y_K^n - x_K| + |y_L^n - x_L| + d_\sig\right) \leq \cter{cte:reg}  {\|\grad \varphi\|}_\infty d_\sig.
\]
with $\ctel{cte:reg} = 1 + 2 \zeta_\Tt$.
Therefore, multiplying~\eqref{eq:cons.1} by $\tau \varphi_K^n$ and summing over $K\in\Tt$ and $n\geq 1$ provides 
\begin{align*}
\sum_{n\geq 1} \sum_{K\in\Tt} m_K \left(\rho_K^n - \rho_K^{n-1} \right) \varphi_K^{n} =\; &
- \sum_{n= 1}^{\lfloor T/\tau\rfloor} \tau \sum_{\sig \in \Ee_\text{int}} m_\sig F_{K\sig}^n \left( \varphi_K^n - \varphi_L^n\right). \\
\leq \; &\cter{cte:reg}  {\|\grad \varphi\|}_\infty \sum_{n= 1}^{\lfloor T/\tau\rfloor} \tau  \sum_{\sig \in \Ee_\text{int}}  a_\sig \eta_\sig^n  \frac{d_\sig \left| F_{K\sig}^n\right|}{\eta_\sig^n} d_\sig, 
\end{align*}
so that a Young-Fenchel inequality gives 
\be\label{eq:dtrho.1}
\sum_{n\geq 1} \sum_{K\in\Tt} m_K \left(\rho_K^n - \rho_K^{n-1} \right) \varphi_K^{n}  \leq \cter{cte:reg}   {\|\grad \varphi\|}_\infty  \sum_{n= 1}^{\lfloor T/\tau\rfloor} \tau  \left(
\Dd_\Ee(\brho^n, \bF^n) + \sum_{\sig \in \Ee_\text{int}}  a_\sig \eta_\sig^n  \Psi^*(d_\sig) \right).
\ee
A Taylor expansion of $\Psi^*$ around $0$ shows that 
\[
\Psi^*(d_\sig)  = \frac{d_\sig^2}2 \left(\Psi^*\right)''(c_\sig) \quad \text{with}\; c_\sig \in (0,d_\sig) \subset [0, \text{diam}(\O)], 
\]
whence, since $\eta_\sig^n \leq 1/4$, 
\be\label{eq:dtrho.2}
\sum_{\sig \in \Ee_\text{int}}  a_\sig \eta_\sig^n  \Psi^*(d_\sig) \leq \frac18 \cosh(\frac{\text{diam}(\O)}2) \sum_{\sig \in \Ee_\text{int}} m_\sig d_\sig 
\leq \frac d8 \cosh(\frac{\text{diam}(\O)}2)m_\O.
\ee
Then we deduce from Proposition~\ref{prop:NRG.disc} and Corollary~\ref{coro:NRG.disc} that 
\be\label{eq:dtrho.3}
 \sum_{n= 1}^{\lfloor T/\tau\rfloor} \tau \Dd_\Ee(\brho^n, \bF^n) \leq \Ff_{\Tt}(\brho^0) -  \Ff_{\Tt,\text{tot}}^{\lfloor T/\tau\rfloor} \leq 
 m_\O \left( \log 2 + \|\phi\|_\infty + 2 \|\grad \phi\|_\infty \delta_\Tt\right) + \cter{cte:Ftot} T. 
\ee
The combination of~\eqref{eq:dtrho.2}\eqref{eq:dtrho.3} in \eqref{eq:dtrho.1} shows the desired result. 
\end{proof}

\section{Convergence analysis}\label{sec:convergence}

The goal of this section is to prove Theorem~\ref{thm:conv}. The proof consists in three steps. 
First in Section~\ref{ssec:compact}, we establish some compactness results on 
$\left(\rho_{\Tt_m, \tau_m}\right)_{m\geq 0}$. Then we identify in Section~\ref{ssec:identify} 
any limit value $\rho$ of $\left(\rho_{\Tt_m, \tau_m}\right)_{m\geq 0}$ as a weak solution to the continuous problem. 
Finally, the uniqueness of the weak solution is established in Section~\ref{ssec:uniqueness}, 
implying by the way the convergence of the whole sequence. 

In what follows, we lighten the notation  by removing the index $m$ associated to the mesh and time step. 
The limit $m\to+\infty$ is denoted by $\delta_\Tt, \tau \to 0$ instead. This limit implicitly supposes that 
the regularity factor of the mesh $\zeta_\Tt$ remains uniformly bounded by some $\zeta_\star$ as prescribed by~\eqref{eq:mtoinfty}. 

\subsection{Compactness properties}\label{ssec:compact}

We derived in Section~\ref{ssec:NRG.disc} all the preliminary material required to use some existing compactness 
results. First by combining Lemmas~\ref{lem:dissip*} and \ref{lem:dissip}, one can apply the black-box discrete 
Aubin-Simon theorem \cite[Theorem 3.9]{ACM17}, leading to the following compactness result. 

\begin{prop}\label{prop:compact}
Let $\rho$ be a limit value~\eqref{eq:rho} of $\rho_{\Tt,\tau}$ as $\delta_\Tt$, $\tau$ tend to $0$, 
then $\rho \in L^2_\text{loc}(\R_+;H^1(\O))$ and, up to a subsequence, 
\be\label{eq:conv.Lp}
\rho_{\Tt,\tau}  \underset{\delta_\Tt, \tau \to 0}\longrightarrow \rho \quad\text{in}\; L^p_\text{loc}(\R_+\times\O).
\ee
\end{prop}

The above proposition shows some strong convergence in the bulk domain $\R_+ \times \O$. 
To pass to the limit in the boundary conditions, one also has to get some convergence of the traces 
on $\R_+\times \p\O$. 
Even though the boundary condition~\eqref{eq:boundary.1} in linear w.r.t. $\rho$, we establish the 
strong convergence of the trace of the approximate solution $\rho_{\Tt,\tau}$ towards 
the trace of $\rho$.

\begin{lem}\label{lem:trace}
Let $\rho_{\Tt,\tau}$ be such that the convergence \eqref{eq:conv.Lp} holds. 
Denote by $\gamma \rho_{\Tt,\tau}$ the trace on $\R_+ \times \G$ of the approximate solution $\rho_{\Tt,\tau}$, i.e. 
\[
\gamma \rho_{\Tt,\tau}(t,x) = \rho_K^n \quad \text{for} \; (t, x) \in (t_{n-1}, t_n] \times \sig, \quad \sig \in \Ee_{K,\text{ext}},\; K \in \Tt, 
\]
and by $\gamma \rho \in L^2_\text{loc}(\R_+;H^{1/2}(\G))$ the trace of a limit value $\rho$ of $\rho_{\Tt,\tau}$, then 
\be\label{eq:conv.trace}
\gamma \rho_{\Tt,\tau} \underset{\delta_\Tt, \tau \to 0}\longrightarrow \gamma \rho \quad \text{in}\; L^p_\text{loc}(\R_+ \times \G), \qquad 1 \leq p < +\oo.
\ee
\end{lem}
\begin{proof}
The proof builds on ideas introduced in Section 4.2 of \cite{BCH13}. First, notice that since both $\rho_{\Tt,\tau}$ and 
$\rho$ remain bounded between $0$ and $1$, is suffices to establish the convergence \eqref{eq:conv.trace} 
in $L^2_\text{loc}(\R_+ \times \G)$ to get it all the $L^p_\text{loc}(\R_+ \times \G)$ thanks to the dominated convergence theorem.

Since $\O$ is assumed to be polyhedral, its boundary $\G$ can be decomposed as $\G = \bigcup_{i=1}^I \G_i$ 
with $\G_i$ included in an hyperplane of $\R^d$ and $I$ finite. We assume that the $\G_i$ are disjointed one from another. 
For $\eps>0$ and $i\in\{1,\dots, I\}$, we define 
\[
\G_{i,\eps} = \{x \in \G_i \; | \; x - \theta \nu_i \in \O \; \text{for}\; \theta \in [0,\eps)\}, 
\]
where $\nu_i$ is the outward w.r.t. $\O$ normal to $\G_i$. Denoting by $m_{\G_{i,\eps}}$ (resp. $m_{\G_i}$)  the 
$(d-1)$-dimensional Hausdorff (or Lebesgue) measure of $\G_{i,\eps}$ (resp. $\G_i$), then 
\[
m_{\G_i} - \cter{cte:mG} \eps \leq m_{\G_{i,\eps}} \leq m_{\G,i}, \qquad \eps >0, \; 1 \leq i \leq I, 
\]
for some $\ctel{cte:mG}$ depending only on $\O$.
Therefore, given an arbitrary final time $T>0$ and an arbitrary $\eps>0$, then for all $i\in\{1,\dots, I\}$, there holds 
\be\label{eq:trace.1}
\int_0^T \int_{\G_i} \left| \gamma \rho_{\Tt,\tau}  - \gamma \rho \right|^2 \leq \int_0^T \int_{\G_{i,\eps}} \left| \gamma \rho_{\Tt,\tau}  - \gamma \rho \right|^2 + \cter{cte:mG} \eps T.
\ee
Using $(a+b+c)^2 \leq 3 (a^2 + b^2 + c^2)$, we obtain that 
\be\label{eq:trace.2}
 \int_0^T \int_{\G_{i,\eps}} \left| \gamma \rho_{\Tt,\tau}(t,y)  - \gamma \rho(t,y) \right|^2 dy dt  \leq A_{\Tt,\tau}^\eps + B_{\Tt,\tau}^\eps  + C^\eps, 
\ee
with 
\begin{align*}
A_{\Tt,\tau}^\eps = \;& \frac3\eps \int_0^T \int_0^\eps \int_{\G_{i,\eps}} \left| \gamma \rho_{\Tt,\tau}(t,y) -  \rho_{\Tt,\tau}(t,y - \theta \nu_i) \right|^2 dy d\theta dt, \\
B_{\Tt,\tau}^\eps = \;& \frac3\eps \int_0^T \int_0^\eps \int_{\G_{i,\eps}} \left| \rho_{\Tt,\tau}(t,y - \theta \nu_i)-  \rho(t,y - \theta \nu_i) \right|^2dy d\theta dt, \\
C^\eps = \;&  \frac3\eps\int_0^T \int_0^\eps \int_{\G_{i,\eps}} \left| \gamma \rho(t,y) -  \rho(t,y - \theta \nu_i) \right|^2 dy d\theta dt. \\
\end{align*}
First, applying Lemma 4.8 of \cite{BCH13} in combination with Lemma~\ref{lem:dissip*} yields 
\be\label{eq:trace.3}
A_{\Tt,\tau}^\eps \leq 3 (\eps + \delta_\Tt) \sum_{n=1}^{\lceil T/\tau \rceil} \tau \sum_{\sig \in \Ee_\text{int}} a_\sig \left(\rho_K^n - \rho_L^n\right)^2 
\leq 3 \cter{cte:L2H1} (1+T+\tau) (\eps + \delta_\Tt).
\ee
Second, it results from Proposition~\ref{prop:compact} that, for any fixed $\eps >0$, there holds 
\be\label{eq:trace.4}
\lim_{\delta_\Tt, \tau \to 0} B_{\Tt,\tau}^\eps = 0. 
\ee
Putting~\eqref{eq:trace.1}--\eqref{eq:trace.4} altogether leads to 
\be\label{eq:trace.5}
\limsup_{\delta_\Tt, \tau \to 0} \int_0^T \int_{\G_i} \left| \gamma \rho_{\Tt,\tau}  - \gamma \rho \right|^2 \leq \big( \cter{cte:mG} T +  3 \cter{cte:L2H1}(1+T) \big) \eps +C^\eps, 
\qquad \forall \eps >0. 
\ee
Eventually, one lets $\eps\to 0$ in~\eqref{eq:trace.5}, the right-hand side of which and in particular $C^\eps$ tend to $0$ since $\gamma \rho$ is the trace of $\rho$. 
This concludes the proof of Lemma~\ref{lem:trace}. 
\end{proof}

Even though the term trace is slightly abusive, it is natural to introduce the alternative notion of trace on $\R_+\times \G$ for the approximate solution $\rho_{\Tt,\tau}$ by setting 
\[
\wt \gamma \rho_{\Tt,\tau}(t,x) = \rho_\sig^n \quad \text{for} \; (t, x) \in (t_{n-1}, t_n] \times \sig, \quad \sig \in \Ee_\text{ext}.
\]
\begin{lem}\label{lem:trace.2}
Let $\rho_{\Tt,\tau}$ be such that the convergence \eqref{eq:conv.Lp} holds, then for all $T>0$, there holds
\be\label{eq:conv.trace.2}
\left\| \gamma  \rho_{\Tt,\tau} - \wt \gamma \rho_{\Tt,\tau}\right\|_{L^p((0,T)\times \G)}  \underset{\delta_\Tt, \tau \to 0}\longrightarrow 0, \qquad 1 \leq p < +\oo.
\ee
In particular, $\wt \gamma \rho_{\Tt,\tau}$ also tends to $\gamma \rho$ in $L^p_\text{loc}(\R_+ \times \G)$ for all finite $p$. 
\end{lem}
\begin{proof}
Once again, the uniform $L^\oo$ bounds on $ \gamma  \rho_{\Tt,\tau}$ and $\wt \gamma  \rho_{\Tt,\tau}$ allow to establish \eqref{eq:conv.trace.2} for $p=1$ only. 
Then, going back to the definitions of $ \gamma  \rho_{\Tt,\tau}$ and $\wt \gamma  \rho_{\Tt,\tau}$, Cauchy-Schwarz inequality gives
\begin{align*}
\left\| \gamma  \rho_{\Tt,\tau} - \wt \gamma \rho_{\Tt,\tau}\right\|_{L^1((0,T)\times \G)}^2  \leq\; & 
 \left(\sum_{n=1}^{\lceil T/\tau \rceil} \tau \sum_{K\in\Tt} \sum_{\sig \in \Ee_{K,\text{ext}}} m_\sig \left| \rho_K^n - \rho_\sig^n \right|\right)^2\\ 
 \leq \; &  
 \left(\sum_{n=1}^{\lceil T/\tau \rceil} \tau \sum_{K\in\Tt} \sum_{\sig \in \Ee_{K,\text{ext}}} a_\sig \left( \rho_K^n - \rho_\sig^n \right)^2 \right) 
  \left(\sum_{n=1}^{\lceil T/\tau \rceil} \tau \sum_{K\in\Tt} \sum_{\sig \in \Ee_{K,\text{ext}}} m_\sig d_\sig \right).
\end{align*}
Thanks to Lemma~\ref{lem:dissip*}, the first term of the right-hand side can be overestimated by 
\[
\sum_{n=1}^{\lceil T/\tau \rceil} \tau \sum_{K\in\Tt} \sum_{\sig \in \Ee_{K,\text{ext}}} a_\sig \left( \rho_K^n - \rho_\sig^n \right)^2
\leq \sum_{n=1}^{\lceil T/\tau \rceil} \tau \sum_{\sig \in \Ee} a_\sig \left( \rho_K^n - \rho_{K\sig}^n \right)^2 \leq \cter{cte:L2H1} (1+ T + \tau).
\]
On the other hand, it follows from the regularity of the mesh that 
\[
\sum_{n=1}^{\lceil T/\tau \rceil} \tau \sum_{K\in\Tt} \sum_{\sig \in \Ee_{K,\text{ext}}} m_\sig d_\sig \leq \zeta_\star \delta_\Tt \sum_{n=1}^{\lceil T/\tau \rceil} \tau \sum_{K\in\Tt} \sum_{\sig \in \Ee_{K,\text{ext}}} m_\sig \leq \zeta_\star \delta_\Tt (T+\tau) m_\G
\]
where $m_\G$ denote the $(d-1)$-dimensional Hausdorff (or Lebesgue) measure of $\G$. 
In particular, \eqref{eq:conv.trace.2} holds for $p=1$, and thus also for all finite $p$. The last statement of the lemma, 
namely the convergence of  $\wt \gamma \rho_{\Tt,\tau}$ towards $\gamma \rho$, is then a straightforward consequence of Lemma~\ref{lem:trace}.
\end{proof}

\subsection{Identification of the limit}\label{ssec:identify}
Our goal is here to establish the consistency of the scheme by identifying any limit value $\rho$ of $\rho_{\Tt,\tau}$ as a solution to the continuous problem. 
\begin{prop}\label{prop:identify}
Let $\rho$ be a limit value of $\rho_{\Tt,\tau}$ as $\delta_\Tt,\tau$ tend to $0$, then $\rho$ is a weak solution 
to the problem~\eqref{eq:cons}--\eqref{eq:init} in the sense of Definition~\ref{def:weak}.
\end{prop}
\begin{proof}
Let $\varphi \in C^\infty_c(\R_+ \times \ov \O)$, then define $\varphi_K^n = \varphi(x_K,t_{n})$ and $\varphi_\sig^n = \varphi(x_\sig,t_{n})$ for all $K \in \Tt$, all $\sig \in \Ee_\text{ext}$ and 
$n \geq 0$. This allows to define the function $\varphi_{\Tt,\tau}$ by 
\[
\varphi_{\Tt,\tau}(t,x) = \varphi_K^{n-1} \quad \text{if}\; (t,x) \in [t_{n-1}, t_n) \times K.
\]
Multiplying~\eqref{eq:cons.1} by $\tau \varphi_{K}^{n-1}$ and summing over $K\in \Tt$ provides 
\be\label{eq:A+B}
A_{\Tt,\tau} + B_{\Tt,\tau} = 0, 
\ee
where we have set 
\[
A_{\Tt,\tau} =  \sum_{n\geq 1} \sum_{K\in\Tt} m_K \left(\rho_K^n - \rho_K^{n-1} \right) \varphi_K^{n-1}
, \qquad 
B_{\Tt,\tau} =  \sum_{n\geq 1} \tau \sum_{K\in\Tt} \sum_{\sig \in \Ee_K} m_\sig F_{K\sig}^n \varphi_K^{n-1}.
\]
Since $\varphi_K^n = 0$ for $n$ large enough, the term $A_{\Tt,\tau}$ can be rewritten as 
\[
A_{\Tt,\tau} =\;   -  \sum_{n\geq 1} \tau \sum_{K\in\Tt} m_K \rho_K^n \frac{\varphi_K^n - \varphi_K^{n-1}}\tau 
- \sum_{K\in\Tt} m_K \rho_K^0 \varphi_K^0.
\]
Then classical arguments (see for instance~\cite{EGH00}) allow to show that 
\be\label{eq:ATttau}
\lim_{\delta_\Tt, \tau \to 0} A_{\Tt,\tau} = - \iint_{\R_+ \times \O} \rho \p_t \varphi - \int_\O \rho^0 \varphi(0).
\ee
On the other hand, thanks to the conservativity of the fluxes, the term $B_{\Tt,\tau}$ reformulates as 
\[
B_{\Tt,\tau} =  \sum_{n\geq 1} \tau \sum_{\sig \in \Ee} m_\sig F_{K\sig}^n \left(\varphi_K^{n-1} - \varphi_{K\sig}^{n-1}\right) 
+  \sum_{n\geq 1} \tau \sum_{K\in\Tt} \sum_{\sig \in \Ee_{K,\text{ext}}} m_\sig F_{K\sig}^n \varphi_\sig^{n-1} 
=: B_{\Tt,\tau}^\text{bulk} + B_{\Tt,\tau}^\text{ext}.
\]
Using the expression of the boundary fluxes~\eqref{eq:FKsig.ext} in the term $B_{\Tt,\tau}^\text{ext}$ provides 
\[
B_{\Tt,\tau}^\text{ext} = 
 \iint_{\R_+\times \G} \left(\a_\Ee \wt \gamma \rho_{\Tt,\tau} - \b_\Ee \right) \wt \gamma \varphi_{\Tt,\tau}
\]
where  $\a_\Ee$ and $\b_\Ee$ are the piecewise constant (per edges $\sig \in \Ee_\text{ext}$) reconstructions on $\G$ build from the evaluation of $\a$ and $\b$ at $x_\sig$, 
and where 
\[
\wt \gamma \varphi_{\Tt, \tau}(t,x) = \varphi_\sig^{n-1} \quad \text{if}\; (t,x) \in [t_{n-1}, t_n) \times \sig, \; \sig \in \Ee_\text{ext}.
\]
Due to the Lipschitz regularity of $\a,\b$ and $\varphi$, their approximations 
 $\a_\Ee$, $\b_\Ee$ and $\wt \gamma \varphi_{\Tt,\tau}$ converge uniformly. 
 One concludes from the convergence of $\wt\gamma \rho_{\Tt,\tau}$ stated at Lemma~\ref{lem:trace.2} that 
 \be\label{eq:BTttau.ext}
\lim_{\delta_\Tt, \tau \to 0} B_{\Tt,\tau}^\text{ext} = 
 \iint_{\R_+\times \G} \left(\a \, \gamma \rho - \b \right)  \varphi.
\ee
For the term $B_{\Tt,\tau}^\text{bulk}$, we use the expression~\eqref{eq:FKsig.int.3} of the internal fluxes, leading to 
\be\label{eq:BTttau.int}
B_{\Tt,\tau}^\text{bulk} = B_{\Tt,\tau}^\text{diff} + B_{\Tt,\tau}^\text{conv} + R_{\Tt,\tau},
\ee
with 
\begin{align*}
B_{\Tt,\tau}^\text{diff} =\;& \sum_{n\geq 1} \tau \sum_{\sig \in \Ee} a_\sig (\rho_K^n - \rho_{K\sig}^n)(\varphi_K^{n-1} - \varphi_{K\sig}^{n-1}), \\
B_{\Tt,\tau}^\text{conv} =\;& \sum_{n\geq 1} \tau \sum_{\sig \in \Ee} a_\sig \frac{\eta(\rho_K^n)+\eta(\rho_{K\sig}^n)}2(\phi_K - \phi_{K\sig}) (\varphi_K^{n-1} - \varphi_{K\sig}^{n-1}), \\
R_{\Tt,\tau} =\;& \sum_{n\geq 1} \tau \sum_{\sig \in \Ee} m_\sig R_\sig^n (\varphi_K^{n-1} - \varphi_{K\sig}^{n-1}).
\end{align*}
We do not detail the proof of 
\be\label{eq:BTttau.diff}
B_{\Tt,\tau}^\text{diff} \underset{\delta_\Tt, \tau \to 0} \longrightarrow  \iint_{\R_+\times\O} \grad \rho \cdot \grad \varphi, 
\qquad 
B_{\Tt,\tau}^\text{conv} \underset{\delta_\Tt, \tau \to 0} \longrightarrow  \iint_{\R_+\times\O} \eta(\rho)\grad \phi \cdot \grad \varphi,
\ee
since similar terms have been studied in many contributions, see for instance~\cite{CCFG21} and references therein. 
It remains to show that $R_{\Tt,\tau}$ vanishes at the limit. We deduce from the expression~\eqref{eq:Rsign}, from the fact that $r_\sig^n \in (0,1)$, and from $\|\eta\|_\infty = 1/4$ that 
\begin{multline*}
\left|R_\sig^n \right| \leq \frac12 \left\|\grad \phi \right\|_\infty (\rho_K^n - \rho_{K\sig}^n)^2 \\
+ \frac{d_\sig} 8  \left\|\grad \phi \right\|_\infty^2 \left( \left|\rho_K^n - \rho_{K\sig}^n\right| \cosh \left\|\phi\right\|_\infty
+ d_\sig \left\|\grad \phi \right\|_\infty \left(\frac12 + (\rho_K^n - \rho_{K\sig}^n)^2\right) \frac12\cosh
\|\phi \|_\infty
\right). 
\end{multline*}
Therefore, using furthermore Lemma~\ref{lem:dissip*}, one readily shows that 
\be\label{eq:RTttau}
R_{\Tt,\tau} \leq C \delta_{\Tt}^2  \underset{\delta_\Tt, \tau \to 0}\longrightarrow 0. 
\ee
Putting~\eqref{eq:ATttau}--\eqref{eq:RTttau} together in~\eqref{eq:A+B} concludes the proof of Proposition~\ref{prop:identify}.
\end{proof}

\subsection{Uniqueness of the weak solution}\label{ssec:uniqueness}

So far, we established the convergence of the scheme towards a weak solution up to a subsequence. 
In order to show that the whole sequence converges, it suffices to show that the limit value is unique. 
This is a consequence of the following proposition. 

\begin{prop}\label{prop:unique}
The weak solution $\rho$ to~\eqref{eq:cons}--\eqref{eq:init} in the sense of Definition~\ref{def:weak} is unique.
\end{prop}
\begin{proof}
Let $\rho$ and $\check \rho$ be two weak solutions corresponding to the same initial data $\rho^0$, and let $T$ be an arbitrary time horizon, then subtracting 
their respective weak formulations leads to 
\be\label{eq:unique.1}
\int_0^T \langle \p_t (\rho-\check \rho) , \varphi \rangle_{H^{-1},H^1_0}  + \int_0^T\int_{\O} \left( \left(\eta(\rho) - \eta(\check \rho)\right) \grad \phi + \grad (\rho - \check \rho) \right) \cdot \grad \varphi = 0
\ee
for all $\varphi \in  L^2((0,T);H^1_0(\O))$.
Choose $\varphi$ as the solution to 
\[
- \Delta \varphi(t) = \rho(t) - \check \rho(t) \quad \text{in}\; \O, \qquad 
\varphi(t) = 0 \quad \text{on}\; \G, \qquad t \in [0,T],
\]
then one readily checks that $\|\grad \varphi(t)\|_{L^2(\O)^d} = \|\rho(t) - \check \rho(t)\|_{H^{-1}(\O)}$. Moreover, 
$\p_t \varphi$ also belongs to $L^2((0,T);H^1_0(\O))$ since $\p_t(\rho-\check \rho)$ belongs to $L^2((0,T);H^{-1}(\O))$.
Therefore, 
\[
\int_0^T \langle \p_t (\rho-\check \rho) , \varphi \rangle_{H^{-1},H^1_0} = \int_0^T \int_\O \p_t \grad \varphi \cdot \grad \varphi
=  \frac12 \| \grad \varphi(T) \|_{(L^2(\O))^d}^2 
\]
since $\varphi(0) = 0$. As a consequence, 
 \eqref{eq:unique.1} yields 
\begin{align*}
\frac12  \|\rho(T) - \check \rho(T)\|_{H^{-1}(\O)}^2 + \| \rho - \check \rho\|_{L^2((0,T)\times \O)}^2 
=&\; -  \int_0^T\int_{\O}  \left(\eta(\rho) - \eta(\check \rho)\right) \grad \phi \cdot \grad \varphi \\
\leq&\; \|\grad \phi\|_\infty  \int_0^T\int_{\O} |\rho-\check\rho| |\grad \varphi|.
\end{align*}
Then we deduce from Young's inequality that 
\[ 
\|\rho(T) - \check \rho(T)\|_{H^{-1}(\O)}^2 \leq  \frac{\|\grad \phi\|_\infty^2}2 \|\rho - \check \rho\|_{L^2((0,T);H^{-1}(\O))}^2.
\]
The above inequality holds for all $T\geq 0$, and we deduce from Gronwall Lemma together and from the fact that $\rho(0) = \check \rho(0) = \rho^0$ that $\|\rho(T) - \check \rho(T)\|_{H^{-1}(\O)} = 0$ for all $T \geq 0$. 
\end{proof}
\section{Numerical results}\label{sec:numerics}

Before presenting numerical results, let us comment briefly on some practical details concerning the effective implementation. Our code in based on Matlab. The resolution of the nonlinear system~\eqref{eq:cons.1}--\eqref{eq:FKsig}, 
{\magenta in its} compact form~\eqref{eq:bHh} is achieved thanks to Newton's method: 
\be\label{eq:Newton}
 \mathbb J(\brho^{n,\ell}) \delta \brho^{n,\ell} = - \Hh^n(\brho^{n,\ell}), \qquad \brho^{n,\ell+1} = \brho^{n,\ell} + \delta \brho^{n,\ell}, 
\ee
with $ \mathbb J$ standing for the Jacobian matrix of $\Hh^n$. Note that $\rho_\sig^n$, $\sig \in \Ee_\text{ext}$ is not considered as an unknown and is deduced from the cell values thanks to~\eqref{eq:rhosign}. 
We initialize~\eqref{eq:Newton} by setting $\brho^{n,0} = \brho^{n-1}$ and than iterate until $\|\delta \brho^{n,\ell}\|_\infty / \| \brho^{n,\ell+1}\|_\infty \leq 10^{-12}$. Then we set $\rho^{n} = \rho^{n,\ell+1}$. 

\subsection{Numerical evidence of the convergence}\label{ssec:num.conv}

The first numerical test we propose aims at confirming our intuition concerning the second order accuracy in space of the scheme sketched in Section~\ref{ssec:scheme}. To this end, we consider a one-dimensional domain $\O = (0,1)$. 
We consider a slightly more general case than the one addressed in the paper by introducing some parameter $\epsilon>0$ (referred later on as the inverse P\'eclet number) in front of the diffusion term in~\eqref{eq:cons.02}:
\be\label{eq:cons.eps}
F  + \eta(\rho) \p_x \phi + \epsilon \p_x \rho = 0. 
\ee
The bulk numerical flux formula~\eqref{eq:FKsig.int} is tuned into
\be\label{eq:FKsig.int.eps}
F_{K\sig}^n =  \frac\epsilon{d_\sig} \left[ \rho_K^n(1-\rho_L^n) e^{\frac1{2\epsilon}(\phi_K - \phi_L)} 
-  \rho_L^n(1-\rho_K^n) e^{\frac1{2\epsilon}(\phi_L - \phi_K)} \right] 
\quad \text{for}\; \sig = K|L \in \Ee_\text{int}. 
\ee
The boundary condition~\eqref{eq:boundary.1} remains unchanged at the continuous level, yet the discrete 
external fluxes are modified into 
\be\label{eq:FKsig.ext.eps}
F_{K\sig}^n = \frac\epsilon{d_\sig} \left[ \rho_K^n(1-\rho_\sig^n) e^{\frac1{2\epsilon}(\phi_K - \phi_\sig)} 
-  \rho_\sig^n(1-\rho_K^n) e^{\frac1{2\epsilon}(\phi_\sig - \phi_K)} \right] = \a_\sig \rho_\sig^n - \b_\sig, \quad \quad \text{for}\; \sig \in \Ee_\text{int}, 
\ee
with the updated boundary density value
\be\label{eq:rhosign.eps}
\rho_\sig^n = \frac{d_\sig \b_\sig + \epsilon \rho_K^n e^{\frac1{2\epsilon}(\phi_K-\phi_\sig)}}{d_\sig \a_\sig +  \epsilon \rho_K^n e^{\frac1{2\epsilon}(\phi_K-\phi_\sig)} + \epsilon (1-\rho_K^n)  e^{-\frac1{2\epsilon}(\phi_K-\phi_\sig)}}.
\ee
The extension of our analysis to this framework is straightforward for fixed values of $\epsilon>0$. 
In our test case, the functions $\alpha$ and $\beta$ defined on $\G = \{0,1\}$ are chosen constant, with $\alpha = 1$ and $\beta = 1/2$. Concerning the external potential, we set $\phi(x) = 1-x$, so that the drift $\p_x \phi$ is constant. 
As an initial data, we choose 
\[
\rho^0(x) = \begin{cases}
1 & \text{if}\; x < 1/2,\\
0 & \text{otherwise.}
\end{cases}
\]
The domain $\O$ is discretized with a successively refined uniform grid. The final time is set to $T=2$, whereas the time step $\tau=10^{-2}$ remains unchanged, in opposition to the spatial mesh size. A reference solution is computed on a fine grid made of 51200 cells.

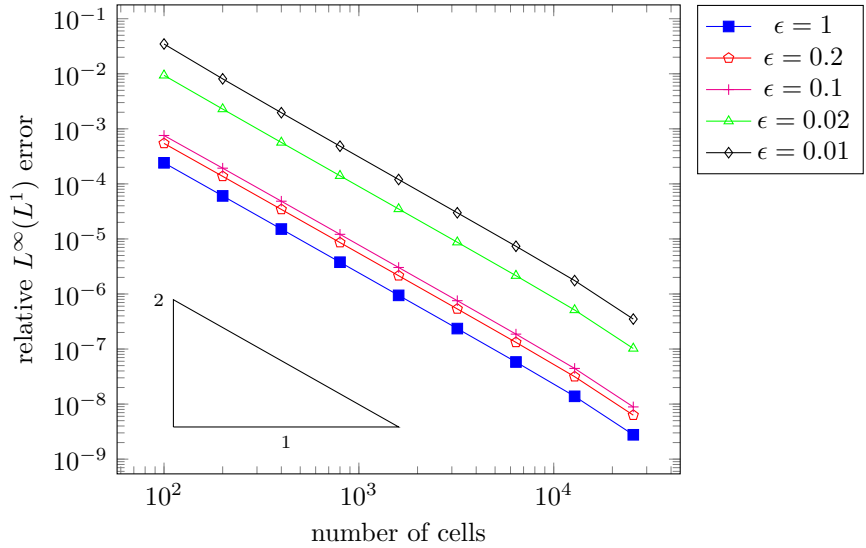
\begin{figure}[htb]
\begin{tikzpicture}
	\begin{loglogaxis}[
	xlabel=number of cells,
	ylabel=relative $L^\oo(L^1)$ error,
	legend style={
		legend pos= outer north east,
	},
	width=0.6\linewidth]
	
	\addplot[mark=square*, color=blue] table[x=NbCells, y=errLinfL1] {num1D/conv_SQRA_1.txt};
	\addlegendentry{$\epsilon = 1$}
	\addplot[mark=pentagon, color=red] table[x=NbCells, y=errLinfL1] {num1D/conv_SQRA_5.txt};
	\addlegendentry{$\epsilon = 0.2$}
	\addplot[mark=+, color=magenta] table[x=NbCells, y=errLinfL1] {num1D/conv_SQRA_10.txt};
	\addlegendentry{$\epsilon = 0.1$}
	\addplot[mark=triangle, color=green] table[x=NbCells, y=errLinfL1] {num1D/conv_SQRA_50.txt};
	\addlegendentry{$\epsilon = 0.02$}
	\addplot[mark=diamond, color=black] table[x=NbCells, y=errLinfL1] {num1D/conv_SQRA_100.txt};
	\addlegendentry{$\epsilon = 0.01$}
	\logLogSlopeTriangle{0.1}{-0.4}{0.1}{2}{black};
	\end{loglogaxis}
\end{tikzpicture}
\caption{Evolution $L^\infty((0,T);L^1(\O))$ relative errors as a function of the number of cells in the spatial discretization for various inverse Péclet numbers $\epsilon$.}
\label{fig:conv}
\end{figure}

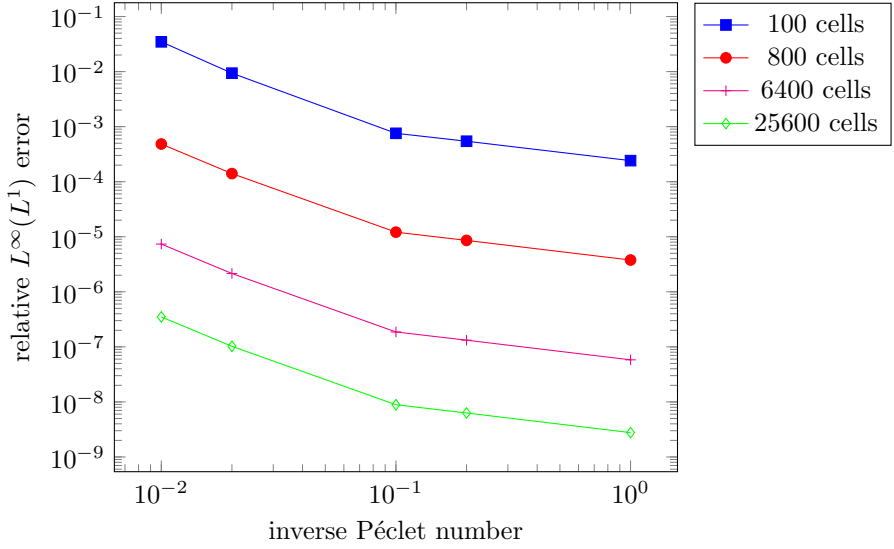
\begin{figure}[htb]
\begin{tikzpicture}
	\begin{loglogaxis}[
	xlabel=inverse Péclet number,
	ylabel=relative $L^\oo(L^1)$ error,
	legend style={
		legend pos= outer north east,
	},
	width=0.6\linewidth]
	
	\addplot[mark=square*, color=blue] table[x=invPeclet, y=errLinfL1] {num1D/conv_SQRA_Peclet_100.txt};
	\addlegendentry{100 cells}
	\addplot[mark=*, color=red] table[x=invPeclet, y=errLinfL1] {num1D/conv_SQRA_Peclet_800.txt};
	\addlegendentry{800 cells}
	\addplot[mark=+, color=magenta] table[x=invPeclet, y=errLinfL1] {num1D/conv_SQRA_Peclet_6400.txt};
	\addlegendentry{6400 cells}
	\addplot[mark=diamond, color=green] table[x=invPeclet, y=errLinfL1] {num1D/conv_SQRA_Peclet_25600.txt};
	\addlegendentry{25600 cells}
	\end{loglogaxis}
\end{tikzpicture}
\caption{Evolution $L^\infty((0,T);L^1(\O))$ on different meshes depending on the inverse Péclet number $\epsilon$.}
\label{fig:Peclet}
\end{figure}

We illustrate on Figure~\ref{fig:conv} the second order convergence in space that was expected from the discussion of Section~\ref{ssec:scheme}. One notices that the error increases when the inverse Péclet number decreases. To better illustrate this point, we plot on Figure~\ref{fig:Peclet} the evolution of the error as a function of $\epsilon$. Such a behavior is expected since the scheme is not asymptotic preserving in the sense that the scheme corresponding to the limit $\epsilon=0$ is not consistent with the limiting hyperbolic continuous equation.

\subsection{Energy stability and long-time behavior}\label{ssec:num.NRG}

Our second numerical experiment is performed on a 2D Delaunay mesh made of 7374 triangles. 
Our goal is here twofold. First, we give a numerical evidence of the fact that the total energy $\Ff_\text{tot}$ decreases along time, while the bulk energy $\Ff(\rho)$ remains bounded. As in Section~\ref{ssec:num.conv}, we introduce the inverse P\'eclet number $\epsilon$. The energy has to  be adapted accordingly by setting 
\[\Ff(\rho) = \int_\O (\epsilon h(\rho) + \rho \phi)\]
with $\phi(x) = 1 - x_2$ for $x=(x_1,x_2) \in \O$. As an initial data, we choose $\rho^0(x) = 1$ if $x \in (0,1/2) \times (0,1/2)$ and $\rho^0(x) = 0$ otherwise.

{\magenta
Two sets of boundary conditions are considered in this section. 
\begin{itemize}
\item First, we fix $\alpha$ and $\beta$ so that there exists some thermal equilibrium. More precisely, we set 
\be\label{eq:ab_eq}
\alpha = 1+ e^{-\frac{\phi-1/2}\epsilon} \quad\text{and}\quad \beta = e^{-\frac{\phi-1/2}\epsilon}.
\ee
The corresponding thermal equilibrium is then given by 
\be\label{eq:rho_eq}
\rho^\infty = \frac{e^{-\frac{\phi-1/2}\epsilon}}{1+e^{-\frac{\phi-1/2}\epsilon}} .
\ee
The inverse Péclet number $\epsilon$ is set to $0.1$. 
\item Second, we choose generic $\alpha$ and $\beta$, for which no thermal equilibrium can be found:
\be\label{eq:ab_noneq}
\alpha \equiv 1, \quad \beta(x) = \frac1{10} + \frac{4}{5}\left( \cos^2\left(\frac{3\pi x_2}{2}\right)+
        (2 x_2 - 1) \sin\left(\pi x_1\right)\right)
, \qquad x = (x_1,x_2) \in \O. 
\ee
Here, we set $\epsilon = 0.01$.
\end{itemize}
}

{\magenta
Let us first address the equilibrium case~\eqref{eq:ab_eq}. Let $\brho^\infty$ be the discrete thermal equilibrium  as in Remark~\ref{rem:equilibrium}, i.e. 
\[
\rho_K^\infty = \frac{e^{-\frac{\phi_K-1/2}\epsilon}}{1+e^{-\frac{\phi_K-1/2}\epsilon}}, \qquad K\in\Tt, 
\]
and denote by $\rho_\Tt^\infty$ the approximate steady state defined by $\rho_\Tt^\infty(x) = \rho_K^\infty$ if $x \in K$.
Then Figure~\ref{fig:exp_eq} exhibits the exponential convergence of $\rho_{\Tt,\tau}$ towards $\rho_\Tt^\infty$. 
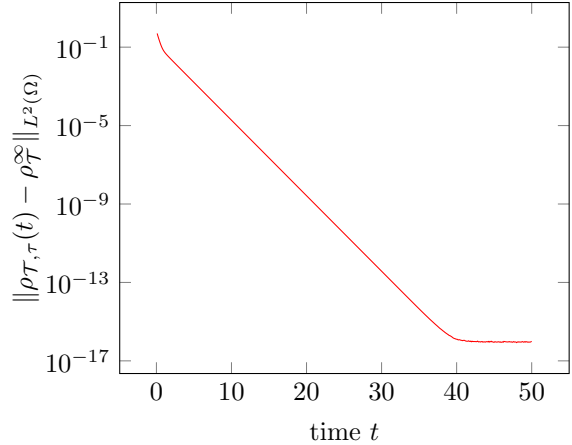
\begin{figure}[htb]
\begin{tikzpicture}
	\begin{semilogyaxis}
	[
	xlabel= time $t$,
	ylabel= $\| \rho_{\Tt,\tau}(t)-\rho_{\Tt}^\infty\|_{L^2(\O)}$,
	width=0.5\linewidth]
	
	\addplot[color=red] table[x=time, y=errL2] {num2D/TempsLong_thermal.txt};
	\end{semilogyaxis}
\end{tikzpicture}
\caption{Evolution of the $L^2$-distance between $\rho_{\Tt,\tau}(t,\cdot)$ and $\rho_{\Tt}^\infty$ as a function of $t$ {\magenta-- equilibrium case~\eqref{eq:ab_eq}}.}
\label{fig:exp_eq}
\end{figure}
}

{\magenta
We now turn to the case of non-equilibrium boundary conditions~\eqref{eq:ab_noneq}.}
Snapshots of the solution are presented on Figure~\ref{fig:snapshots}.
\begin{figure}[htb]
\begin{tabular}{ccc}
\includegraphics[width=.3\textwidth]{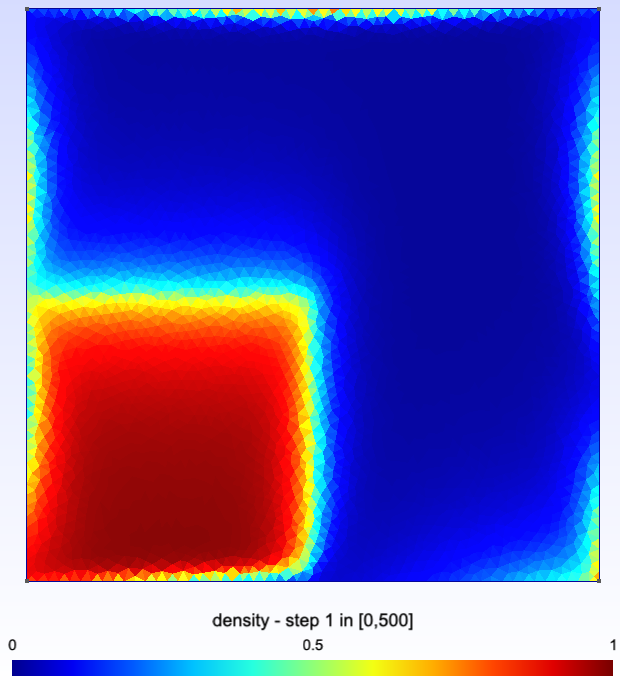} &
\includegraphics[width=.3\textwidth]{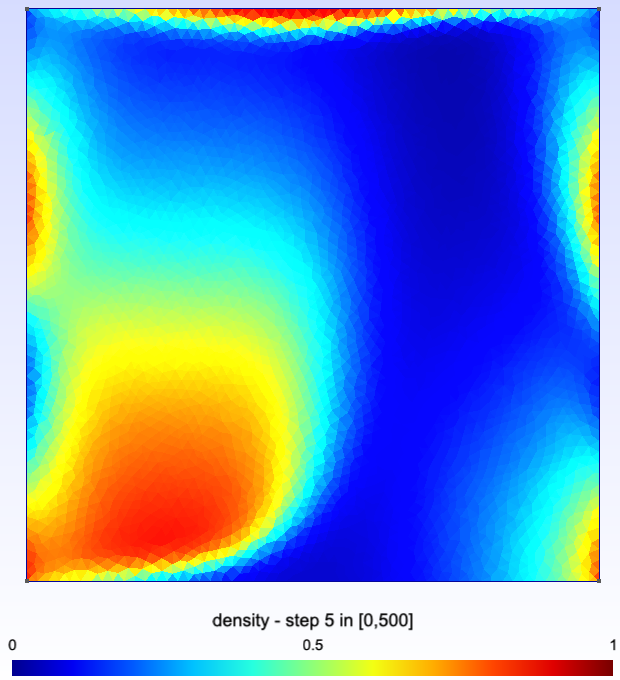} &
\includegraphics[width=.3\textwidth]{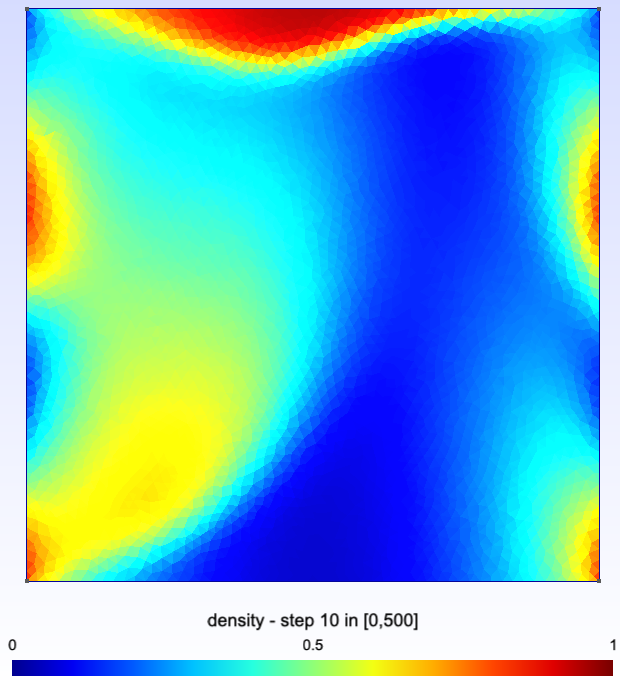} \\
$t=0.1$ & $t=0.5$ & $t=1$\\[10pt]
\includegraphics[width=.3\textwidth]{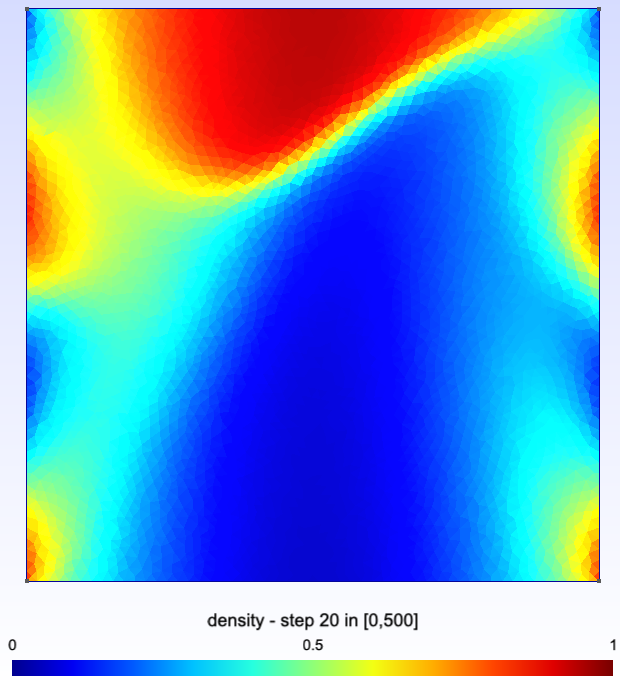} &
\includegraphics[width=.3\textwidth]{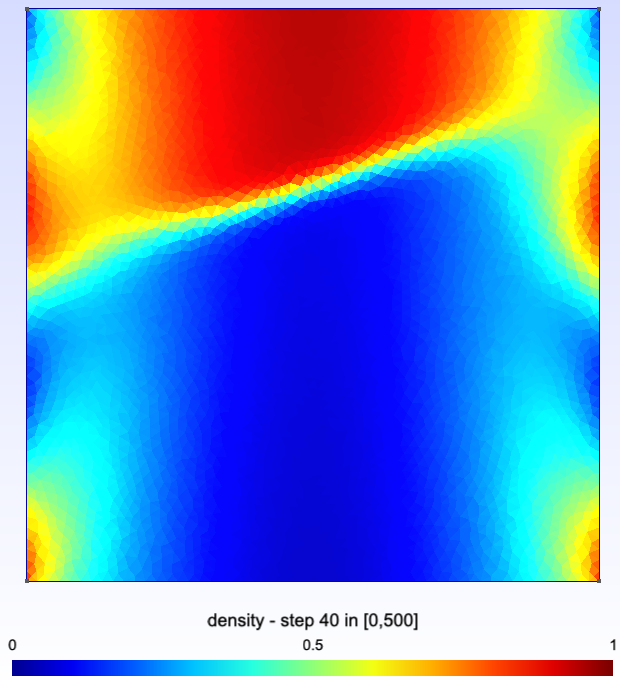} &
\includegraphics[width=.3\textwidth]{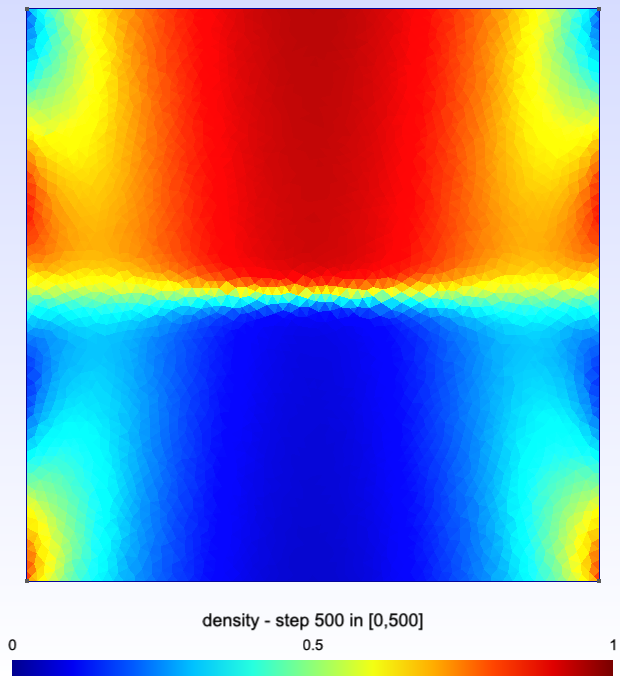} \\
$t=2$ & $t=4$ & $t=50$
\end{tabular}
\caption{Snapshots of the solution at different times {\magenta-- non-equilibrium case~\eqref{eq:ab_noneq}}.} \label{fig:snapshots}
\end{figure}

We plot on Figure~\ref{fig:NRG} the evolution of the bulk and total energies along time. As expected, $\Ff_\text{tot}$ is decreasing with linear decay, while $\Ff(\rho)$ remains bounded along time. 
\begin{figure}[htb]
\begin{tikzpicture}
	\begin{axis}
	[
	xlabel= time $t$,
	ylabel=total energy $\Ff_\text{tot}$,
	width=0.4\linewidth]
	
	\addplot[color=red] table[x=time, y=NRG_tot] {num2D/NRG.txt};
	\end{axis}
\end{tikzpicture}
\hspace{.1\linewidth}
\begin{tikzpicture}
	\begin{axis}
	[
	xlabel= time $t$,
	ylabel=internal energy $\Ff(\rho)$,
	width=0.4\linewidth]
	
	\addplot[color=red] table[x=time, y=NRG_int] {num2D/NRG.txt};

	\end{axis}
\end{tikzpicture}
\caption{Evolution of the total energy (left) and of the bulk energy (right) along time {\magenta-- non-equilibrium case~\eqref{eq:ab_noneq}}.}
\label{fig:NRG}
\end{figure}
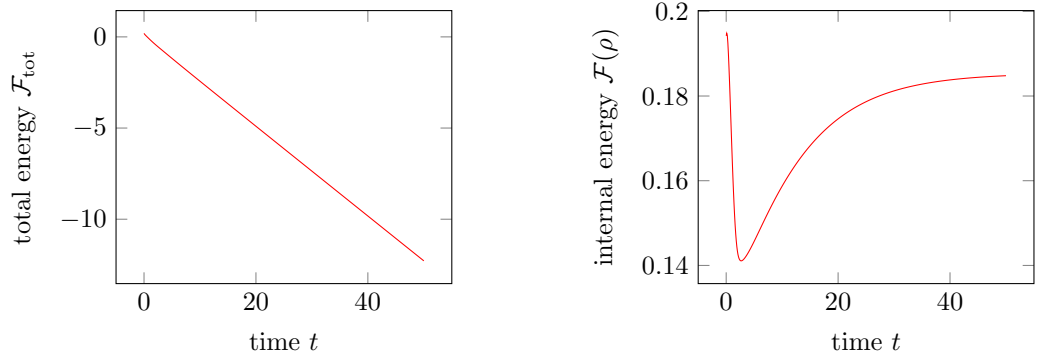

We make use of a uniform time step $\tau=0.1$ until we reach the final time {\magenta $T=200$. Then the steady longtime limit $\brho^\infty$ corresponding to 
$t=10^4$  is computed with larger time step $\tau = 100$. Even though there is no thermal equilibrium for the test-case under consideration, the numerical solution 
still exponentially converges towards the steady state, as shows Figure~\ref{fig:exp_noneq}. The nonlinearity of our problem~\eqref{eq:cons}--\eqref{eq:init} does not enter 
the framework proposed in~\cite{BLMV14}, the extension of which to the discrete setting~\cite{FH17, CHH20} do not apply directly. The proof of the exponential convergence of the scheme towards non-equilibrium steady states should be addressed in future works.
\begin{figure}[htb]
\begin{tikzpicture}
	\begin{semilogyaxis}
	[
	xlabel= time $t$,
	ylabel= $\| \rho_{\Tt,\tau}(t)-\rho_\Tt^\infty\|_{L^2(\O)}$,
	width=0.5\linewidth]
	
	\addplot[color=red] table[x=time, y=errL2] {num2D/TempsLong.txt};
	\end{semilogyaxis}
\end{tikzpicture}
\caption{Evolution of the $L^2$-distance between $\rho_{\Tt,\tau}(t,\cdot)$ and $\rho_{\Tt}^\oo$ as a function of $t$ {\magenta-- non-equilibrium case~\eqref{eq:ab_noneq}}.}
\label{fig:exp_noneq}
\end{figure}
}

{\magenta Finally} we highlight the good behavior of the numerical scheme when it comes to the effective resolution of the induced nonlinear system. 
 As expected, the highest number of required Newton iteration corresponds to the initial time steps where only 17 Newton iterations are required although $\brho^1$ significantly differs from $\brho^0$. 
 As time goes, this number decreases. In our test case, the steady state is not yet reached for $T=50$ and still $9$ Newton iterations per time step 
 are needed to solve the nonlinear system. This number can be importantly decreased for less demanding stopping criteria. 
The number of required Newton iterations at each time step is  reported on Figure~\ref{fig:Newton}. 
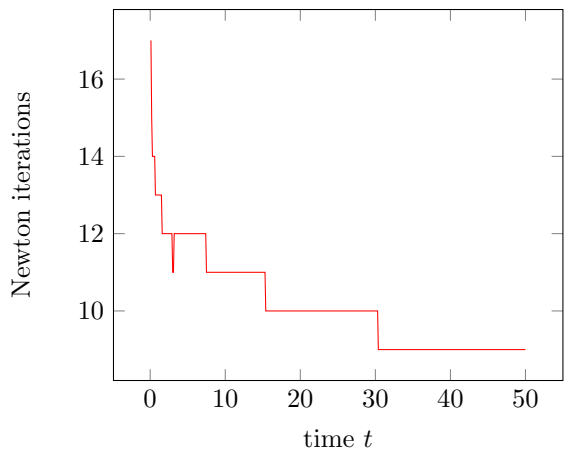
\begin{figure}[htb]
\begin{tikzpicture}
	\begin{axis}
	[
	xlabel= time $t$,
	ylabel= Newton iterations,
	width=0.5\linewidth]
	
	\addplot[color=red] table[x=time, y=iterations] {num2D/Newton.txt};

	\end{axis}
\end{tikzpicture}
\caption{Number of Newton iterations required to solve the nonlinear system at each time step {\magenta-- non-equilibrium case~\eqref{eq:ab_noneq}}.}\label{fig:Newton}
\end{figure}

\subsection*{Acknowledgements} This project has received funding from the European Union's Horizon 2020 research 
and innovation programme under grant agreement No 847593 (EURAD program, WP DONUT), 
and was further supported by Labex CEMPI (ANR-11-LABX-0007-01) 
and the Fédération de Recherche Mathématique des Hauts-de-France (FR CNRS 2037). 
C. Cancès also acknowledges support from the COMODO (ANR-19-CE46-0002) and MICMOV (ANR-19-CE40-0012) projects. 
J. Venel warmly thanks the Inria research center of the University of Lille for its hospitality and support, and the authors 
further thank Claire Chainais for stimulating discussions.


\begin{thebibliography}{10}

\bibitem{Ahmed_intrusion}
A.~Ait Hammou~Oulhaj.
\newblock Numerical analysis of a finite volume scheme for a seawater intrusion
  model with cross-diffusion in an unconfined aquifer.
\newblock {\em Numer. Methods Partial Differential Equations}, 34(3):857--880,
  2018.

\bibitem{ACM17}
B.~Andreianov, C.~Canc\`es, and A.~Moussa.
\newblock A nonlinear time compactness result and applications to
  discretization of degenerate parabolic–elliptic {PDE}s.
\newblock {\em J. Funct. Anal.}, 273(12):3633--3670, 2017.

\bibitem{BCL16}
J.-D. Benamou, G.~Carlier, and M.~Laborde.
\newblock An augmented {L}agrangian approach to {W}asserstein gradient flows
  and applications.
\newblock In {\em Gradient flows: from theory to application}, volume~54 of
  {\em ESAIM Proc. Surveys}, pages 1--17. EDP Sci., Les Ulis, 2016.

\bibitem{BLMV14}
T.~Bodineau, J.~Lebowitz, C.~Mouhot, and C.~Villani.
\newblock Lyapunov functionals for boundary-driven nonlinear drift--diffusion
  equations.
\newblock {\em Nonlinearity}, 27(9):2111, 2014.

\bibitem{BCH13}
K.~Brenner, C.~Canc{\`e}s, and D.~Hilhorst.
\newblock {Finite volume approximation for an immiscible two-phase flow in
  porous media with discontinuous capillary pressure}.
\newblock {\em Comput. Geosci.}, 17(3):573--597, 2013.

\bibitem{BD10}
C.~Buet and S.~Dellacherie.
\newblock On the {C}hang and {C}ooper scheme applied to a linear
  {F}okker-{P}lanck equation.
\newblock {\em Commun. Math. Sci.}, 8(4):1079--1090, 2010.

\bibitem{BDPS10}
M.~Burger, M.~Di~Francesco, J.-F. Pietschmann, and B.~Schlake.
\newblock Nonlinear cross-diffusion with size-exclusion.
\newblock {\em SIAM J. Math. Anal.}, 46(6):2842--2871, 2010.

\bibitem{CCMRV}
C.~Canc{\`e}s, C.~Chainais-Hillairet, B.~Merlet, F.~Raimondi, and J.~Venel.
\newblock {Mathematical analysis of a thermodynamically consistent reduced
  model for iron corrosion}.
\newblock working paper or preprint, January 2022.

\bibitem{CCFG21}
C.~Cancès, C.~Chainais-Hillairet, J.~Fuhrmann, and B.~Gaudeul.
\newblock A numerical analysis focused comparison of several finite volume
  schemes for a unipolar degenerated drift-diffusion model.
\newblock {\em IMA J. Numer. Anal.}, 41(1):271--314, 2021.

\bibitem{CGT20}
C.~Cancès, T.~O. Gallou\"et, and G.~Todeschi.
\newblock A variational finite volume scheme for {W}asserstein gradient flows.
\newblock {\em Numer. Math.}, 146(3):437--480, 2020.

\bibitem{CCWW_FoCM}
J.~A. Carrillo, K.~Craig, L.~Wang, and C.~Wei.
\newblock Primal dual methods for {W}asserstein gradient flows.
\newblock {\em Found. Comput. Math.}, 2021.
\newblock Online first.

\bibitem{CHH20}
C.~Chainais-Hillairet and M.~Herda.
\newblock Large-time behaviour of a family of finite volume schemes for
  boundary-driven convection--diffusion equations.
\newblock {\em IMA J. Numer. Anal.}, 40(4):2473--2504, 2020.

\bibitem{Chatard_FVCA6}
M.~Chatard.
\newblock Asymptotic behavior of the {S}charfetter-{G}ummel scheme for the
  drift-diffusion model.
\newblock In {\em Finite volumes for complex applications {VI}. {P}roblems \&
  perspectives. {V}olume 1, 2}, volume~4 of {\em Springer Proc. Math.}, pages
  235--243. Springer, Heidelberg, 2011.

\bibitem{Dei85}
K.~Deimling.
\newblock {\em Nonlinear functional analysis}.
\newblock Springer-Verlag, Berlin, 1985.

\bibitem{EFG06}
R.~Eymard, J.~Fuhrmann, and K.~G\"artner.
\newblock A finite volume scheme for nonlinear parabolic equations derived from
  one-dimensional local {D}irichlet problems.
\newblock {\em Numer. Math.}, 102:463--495, 2006.

\bibitem{EGGH98}
R.~Eymard, T.~Gallou{\"e}t, M.~Ghilani, and R.~Herbin.
\newblock Error estimates for the approximate solutions of a nonlinear
  hyperbolic equation given by finite volume schemes.
\newblock {\em IMA J. Numer. Anal.}, 18(4):563--594, 1998.

\bibitem{Tipi}
R.~Eymard, T.~Gallou\"et, C.~Guichard, R.~Herbin, and R.~Masson.
\newblock {TP} or not {TP}, that is the question.
\newblock {\em Comput. Geosci.}, 18:285--296, 2014.

\bibitem{EGH00}
R.~Eymard, T.~Gallou\"et, and R.~Herbin.
\newblock Finite volume methods.
\newblock Ciarlet, P. G. (ed.) et al., in Handbook of numerical analysis.
  North-Holland, Amsterdam, pp. 713--1020, 2000.

\bibitem{FH17}
F.~Filbet and M.~Herda.
\newblock {A finite volume scheme for boundary-driven convection-diffusion
  equations with relative entropy structure}.
\newblock {\em {Numerische Mathematik}}, 137:535–577, 2017.

\bibitem{Frenzel19}
Thomas Frenzel.
\newblock {\em On the derivation of effective gradient systems via
  {EDP}-convergence}.
\newblock PhD thesis, Humboldt-Universit\"at zu Berlin,
  Mathematisch-Naturwissenschaftliche Fakult\"at, 2019.

\bibitem{GK_Voronoi}
K.~G{\"a}rtner and L.~Kamenski.
\newblock Why {{Do We Need Voronoi Cells}} and {{Delaunay Meshes}}?
\newblock In Vladimir~A. Garanzha, Lennard Kamenski, and Hang Si, editors, {\em
  Numerical {{Geometry}}, {{Grid Generation}} and {{Scientific Computing}}},
  Lecture {{Notes}} in {{Computational Science}} and {{Engineering}}, pages
  45--60, {Cham}, 2019. {Springer International Publishing}.

\bibitem{Heida18}
M.~Heida.
\newblock Convergences of the squareroot approximation scheme to the
  {F}okker–{P}lanck operator.
\newblock {\em Math. Models Methods Appl. Sci.}, 28(13):2599--2635, 2018.

\bibitem{HKS21}
M~Heida, M.~Kantner, and A.~Stephan.
\newblock Consistency and convergence for a family of finite volume
  discretizations of the {F}okker-{P}lanck operator.
\newblock {\em ESAIM: M2AN}, 55(6):3017--3042, 2021.

\bibitem{HT_arXiv}
A.~Hraivoronska and O.~Tse.
\newblock Diffusive limit of random walks on tesselations via generalized
  gradient flows.
\newblock arXiv:2202.06024, 2022.

\bibitem{KL99}
C.~Kipnis, C. ;~Landim.
\newblock {\em Scaling limits of interacting particle systems}, volume 320.
\newblock Springer, New-York, grundlehrender mathematischen {W}issenschaften
  edition, 1999.

\bibitem{LS34}
J.~Leray and J.~Schauder.
\newblock Topologie et \'equations fonctionnelles.
\newblock {\em Ann. Sci. \'Ecole Norm. Sup.}, 51((3)):45--78, 1934.

\bibitem{LLW20}
W.~Li, J.~Lu, and L.~Wang.
\newblock Fisher information regularization schemes for {W}asserstein gradient
  flows.
\newblock {\em J. Comput. Phys.}, 416:109449, 2020.

\bibitem{LFW13}
H.~C. Lie, K.~Fackeldey, and M.~Weber.
\newblock A square root approximation of transition rates for a {M}arkov state
  model.
\newblock {\em SIAM J. Matrix Anal. Appl.}, 34(2):738--756, 2013.

\bibitem{MO14}
D.~Matthes and H.~Osberger.
\newblock Convergence of a variational {L}agrangian scheme for a nonlinear
  drift diffusion equation.
\newblock {\em ESAIM Math. Model. Numer. Anal.}, 48(3):697--726, 2014.

\bibitem{Mie11}
A.~Mielke.
\newblock A gradient structure for reaction-diffusion systems and for
  energy-drift-diffusion systems.
\newblock {\em Nonlinearity}, 24(4):1329--1346, 2011.

\bibitem{MPPR17}
A.~Mielke, R.~I.~A. Patterson, M.~A. Peletier, and M.~D.~R. Renger.
\newblock Non-equilibrium thermodynamical principles for chemical reactions
  with mass-action kinetics.
\newblock {\em SIAM J. Appl. Math.}, 77(4):1562--1585, 2017.

\bibitem{MPR14}
A.~Mielke, M.~A. Peletier, and D.~R.~M. Renger.
\newblock On the relation between gradient flows and the large-deviation
  principle, with applications to {M}arkov chains and diffusion.
\newblock {\em Potential Anal.}, 41(4):1293--1327, 2014.

\bibitem{PRST22}
M.~A. Peletier, R.~Rossi, G.~Savar\'e, and O.~Tse.
\newblock Jump processes as generalized gradient flows.
\newblock {\em Calc. Var. Partial Differential Equations}, 61:33, 2022.

\bibitem{PS22}
M.~A. Peletier and A.~Schlichting.
\newblock Cosh gradient systems and tilting.
\newblock {\em Nonlinear Anal.}, page 113094, 2022.
\newblock (Online first, https://doi.org/10.1016/j.na.2022.113094).

\end{thebibliography}
\end{document}